\documentclass[a4paper,12pt]{amsart}

\allowdisplaybreaks

\usepackage{amsthm,amsmath,amssymb,amsxtra}
\usepackage{bm,mathrsfs,marvosym,eurosym}
\usepackage{mathtools}

\usepackage{enumerate,comment,graphicx,psfrag}
\usepackage{xspace}

\usepackage{color}

\usepackage{hhline}

\usepackage{hyperref}

\usepackage{nameref}

\usepackage{verbatim}
\usepackage{tabls}

\def\MR#1{\href{http://www.ams.org/mathscinet-getitem?mr=#1}{MR#1}}

\begin{document}
\theoremstyle{plain}
\newtheorem{theorem}{Theorem}[section]
\newtheorem{lemma}{Lemma}[section]
\newtheorem{proposition}{Proposition}[section]
\newtheorem{corollary}{Corollary}[section]

\theoremstyle{definition}
\newtheorem{definition}[theorem]{Definition}

\newtheorem{example}[theorem]{Example}

\newtheorem{remark}{Remark}[section]
\newtheorem{remarks}[remark]{Remarks}
\newtheorem{note}{Note}
\newtheorem{case}{Case}

\numberwithin{equation}{section}
\numberwithin{table}{section}
\numberwithin{figure}{section}



\def\div{\text{\rm div}}
\def\cal{\Cal } 
\def\d{{\mathrm d}}
\def\tr|{|\! |\! |}

\def\i{\mathrm{i} }
\def\e{\mathrm{e} }
\def\C{{\mathbb C}}
\def\P{{\mathbb P}}
\def\R {{\mathbb R}}
\def\N{{\mathbb N}}
\def\Z{{\mathbb Z}}
\def\N{{\mathbb N}}
\def\T{{\mathbb T}}
\def\Re{\text {\rm Re} }

\def\D{{\mathscr{D}}}

\def\A{{\mathcal{A}}}
\def\E{{\mathcal{E}}}
\def\L{{\mathcal{L}}}
\def\M{{\mathcal{M}}}
\def\X{{\mathcal{X}}}
\def\V{{\mathcal{V}}}
\def\U{{\mathcal{U}}}
\def\W{{\mathcal{W}}}

\def\<{{\langle }}
\def\>{{\rangle }}

\def\wU{\widehat U}
\def\wV{\widehat V}

\def\AA{\text{${\mathcal O}\hskip-3pt\iota$}}

\newcommand{\stig}[1]{{\color{red}{#1}}}



\title[Discontinuous Galerkin methods for nonlinear parabolic equations]
{Error analysis for discontinuous Galerkin time-stepping methods for nonlinear
  parabolic equations via maximal regularity} 

\author[Georgios Akrivis]{Georgios Akrivis}
\address{Department of Computer Science and Engineering, University of Ioannina, 
451$\,$10 Ioannina, Greece, and Institute of Applied and Computational Mathematics, 
FORTH, 700$\,$13 Heraklion, Crete, Greece}
\email {\href{mailto:akrivis@cse.uoi.gr}{akrivis{\it @\,}cse.uoi.gr}}  
\author[Stig Larsson]{Stig Larsson}
\address{Department of Mathematical Sciences, Chalmers University of Technology
and University of Gothenburg, SE--412 96~Gothenburg, Sweden}
\email {\href{mailto:stig@chalmers.se}{stig{\it @\,}chalmers.se}}  
\date{\today}


\thanks{S.~L.~was supported by Vetenskapsr\aa{}det (VR) through grant
  no.~2017-04274.}

\begin{abstract}
  We consider the discretization of a class of nonlinear parabolic
  equations by discontinuous Galerkin time-stepping methods and
  establish a priori as well as conditional a posteriori error
  estimates.  Our approach is motivated by the error analysis in
  \cite{KuLL} for Runge--Kutta methods for nonlinear parabolic
  equations; in analogy to \cite{KuLL}, the proofs are based on
  maximal regularity properties of discontinuous Galerkin methods for
  non-autonomous linear parabolic equations.
\end{abstract}

\maketitle


\section{Introduction}\label{Se:1}

\subsection{The  initial and boundary value problem}\label{SSe:1.1}
We consider an initial and boundary value problem for a nonlinear
parabolic equation, subject to homogeneous Dirichlet boundary
conditions,
\begin{equation}
\label{ivp}
\left \{
\begin{alignedat}{3} 
&u_t=\nabla\cdot f(\nabla u, u) \quad && \text{in } &&\varOmega \times (0,T],\\
&u=0 \quad && \text{on } &&\partial\varOmega \times (0,T],\\
&u(\cdot,0)=u_0\quad && \text{in } &&\varOmega,\\
\end{alignedat}
\right .
\end{equation}
with a given initial value $u_0$ and $\varOmega\subset \R^d$  a bounded domain
with smooth boundary $\partial\varOmega$. We assume that the flux function
$f \colon \R^d\times \R\to \R^d$ is smooth and satisfies a \emph{local} ellipticity condition,
namely, for every $(p,v)\in \R^d\times \R,$ the symmetric part of the
Jacobian matrix $D_pf(p, v)$, that is, the matrix 
\begin{equation}
\label{local_ellipticity}
\frac 12 \big (D_pf(p, v)+D_pf(p, v)^\top\big )\in \R^{d,d}\ \text{is positive definite},
\end{equation}
and that \eqref{ivp}  possesses a unique, sufficiently smooth solution $u.$
Notice that eigenvalues of the symmetric matrix in \eqref{local_ellipticity} are
allowed to tend to $0$ or $\infty$ as $|(p,v)|\to \infty.$

We refer to \cite{KuLL} for several parabolic equations of the form \eqref{ivp}
occurring in applications as well as for previous work on error estimates
for time discretizations of \eqref{ivp}.

\subsection{Discontinuous Galerkin time discretization}\label{SSe:1.2}
We consider the discretization of the initial and boundary value problem \eqref{ivp} 
by discontinuous Galerkin (dG) methods. 

Let $N\in \N, k=T/N$ be the constant time step,
$t_n:=nk,\ n=0,\dotsc,N,$ be a uniform partition of the time interval
$[0,T],$ and $J_n:= (t_n,t_{n+1}].$

For $s\in \N_0,$ we denote by $\P(s)$ the space of polynomials of
degree at most $s$ with coefficients in $H^1_0(\varOmega)$, i.e., the elements
$g$ of $\P(s)$ are of the form
\[g(t)= \sum_{j=0}^s  t^j w_j, \quad w_j\in  H^1_0(\varOmega), \quad j=0,\dotsc, s.\]
With this notation, let $\V_k^{\text{c}} (s)$ and
$\V_k^{\text{d}} (s)$ be the spaces of continuous and possibly
discontinuous piecewise elements of $\P(s)$, respectively,
\[\begin{aligned}
&\V_k^{\text{c}} (s):=\{v\in C\big ([0,T];H^1_0(\varOmega)\big ): v|_{J_n}\in \P(s), \ n=0,\dotsc, N-1\},\\
&\V_k^{\text{d}} (s):=\{v: [0,T]\to H^1_0(\varOmega), \ v|_{J_n}\in \P(s), \ n=0,\dotsc, N-1\}.
\end{aligned}\]
The spaces $\L_k^{\text{c}} (s)$ and $\L_k^{\text{d}} (s)$ are defined
analogously, with coefficients $w_j\in L^2(\varOmega)$.

We denote by $(\cdot,\cdot)$ the inner product both on
$L^2(\varOmega)$ and on $\big (L^2(\varOmega)\big )^d.$

For $q\in \N,$ with starting value $U(0)=U_0=u_0,$ we consider the
discretization of the initial and boundary value problem \eqref{ivp}
by the \emph{discontinuous Galerkin method} dG$(q-1)$, i.e., we seek
$U\in \V_k^{\text{d}} (q-1)$ such that
\begin{equation}
\label{dg}
\int_{J_n}   \big[ ( U_t ,v )  + (f(\nabla U, U) ,\nabla v ) \big] \, \d t  
+ ( U_n^{+}-U_n, v_n^{+})= 0 \quad \forall v \in \P(q-1)
\end{equation}
for $n=0,\dotsc,N-1$.  As usual, we use the notation $v_n:=v(t_n),$
$v_n^{+}:=\lim_{s\searrow 0} v(t_n+s)$. 

As we shall see later on, there exists a locally unique solution
$U(t)\in W^{1,\infty}(\varOmega), 0\leqslant t\leqslant T,$ of
\eqref{dg}; then, $f(\nabla U(t), U(t))\in (L^2(\varOmega))^d$ and the
formulation in \eqref{dg} makes sense.

With $0<c_1<\dotsb<c_q=1$ the Radau nodes in the interval $[0,1]$, let
$t_{ni}:=t_n+c_ik, i=1,\dotsc,q,$ be the intermediate nodes in
$\bar{J}_n=[t_n,t_{n+1}]$; we also define $t_{n0}:=t_n.$
Following \cite{MN}, we define the \emph{reconstruction mapping}
$\V_k^{\text{d}} (q-1) \to \V_k^{\text{c}} (q),$
$w\mapsto \widehat w,$ via extended interpolation at the Radau nodes,
\begin{equation}
\label{wU_def-inter}
\widehat w ( t_{nj})= w( t_{nj}), \quad j=0, \dotsc ,q \quad
\text{(where $w(t_{n0})=w_n$)}.
\end{equation}
Using the Lagrange form of the difference
$\widehat w(t)-w(t)=(w_n-w_n^+)\hat\ell_{n0}(t),$ with
$\hat \ell_{n0}$ the polynomial of degree $q$ vanishing at the nodes
$t_{n1},\dotsc,t_{nq}$ and with the value $1$ at $t_{n0},$ integration
by parts and the fact that the Radau quadrature rule with $q$ nodes
integrates polynomials of degree at most $2q-2$ exactly, it is easily
seen that
\begin{equation}
\label{wU_def}
\int_{J_n}    (\widehat w_t ,v )\,\d t = \int_{J_n}    ( w_t ,v )   \, \d t
+ ( w_n^+-w_n, v_n^+)  \quad \forall v \in \P(q-1) .
\end{equation}
%
Therefore, with the reconstruction $\wU$ of $U,$ we can reformulate
the dG method \eqref{dg} as
\begin{equation}
\label{dgwU}
\int_{J_n}   \big[ ( \wU_t ,v )  + (f(\nabla U, U) ,\nabla v ) \big] \, \d t  
= 0 \quad \forall v \in \P(q-1) .
\end{equation}
Denoting by $P_{q-1}$ the piecewise $L^2$-projection onto
$\L_k^{\text{d}}(q-1)$, the relation \eqref{dgwU} implies the
\emph{pointwise equation}
\begin{equation}
\label{dgwU_pw}
\wU_t -  P_{q-1} \nabla\cdot  f(\nabla U, U) =0,
\end{equation}
which will be used occasionally. 

\subsection{Main results}\label{SSe:1.3}
We establish the following optimal order, optimal regularity a priori and
conditional a posteriori error estimates.

\begin{theorem}[A priori error estimates]\label{Theorem1}
  Let $p$ and $r$ be sufficiently large such that $2/p + d/r < 1,$ and
  assume that the solution of \eqref{ivp} is sufficiently regular,
  namely,
  $u\in W^{q,p}\big ((0,T);W^{2,r}(\varOmega)\cap
  W^{1,r}_0(\varOmega)\big ).$ Then there is $k_0$ such that, for
  $k\leqslant k_0$, there is a neighborhood of $u$ in
  $L^\infty((0,T);W^{1,\infty}(\varOmega))$ such that the dG
  equation~\eqref{dg} has a unique solution
  $U\in \V_k^{\text{d}} (q-1)$ in this neighborhood, and $U$ satisfies
  the convergence estimate
  \begin{equation}
    \label{apriori-estimate1}
    \|u-U\|_{L^p((0,T);W^{2,r}(\varOmega))}  \leqslant C k^q.
  \end{equation}
  Furthermore, if 
   $u\in W^{q+1,p}\big ((0,T);L^r(\varOmega)\big ),$ for the
  reconstruction $\wU\in \V_k^{\text{c}} (q)$ of $U$, we have
  \begin{equation}
    \label{apriori-estimate2}
    \|u_t - \widehat U_t\|_{L^p((0,T);L^r(\varOmega))}
    + \|u-\wU\|_{L^p((0,T);W^{2,r}(\varOmega))} \leqslant C k^q
  \end{equation}
  and
  \begin{equation}
    \label{apriori-estimate3}
    \|u-\widehat U\|_{L^\infty((0,T);W^{1,\infty}(\varOmega))}
    \leqslant C k^q.  
  \end{equation}
  The constant $C$ depends on $q, f, T$, and
  $u$, but it is independent of the time step $k$.
\end{theorem}

Notice that our assumptions $u\in W^{q,p}\big ((0,T);W^{2,r}(\varOmega)\cap
W^{1,r}_0(\varOmega)\big )$  and $q\geqslant 1, p>2$ imply
$u\in C\big( [0,T];W^{2,r}(\varOmega)\cap W^{1,r}_0(\varOmega)\big);$
in particular, $u_0\in W^{2,r}(\varOmega)\cap W^{1,r}_0(\varOmega)$. 

\begin{theorem}[Conditional a posteriori error estimate]\label{Theorem2} 
  Let $p$ and $r$ be sufficiently large such that $2/p + d/r < 1$ and
  let $u$ and $U$ be as in Theorem~\ref{Theorem1}. Assume
  that the reconstruction $\wU\in \V_k^{\text{c}} (q)$ of the dG
  approximation $U\in \V_k^{\text{d}} (q-1)$ is such that
\begin{equation}
\label{conditional}
\|u-\wU\|_{L^\infty((0,T);W^{1,\infty}(\varOmega))}
\leqslant \nu
\end{equation}
with a sufficiently small constant $\nu,$ independent of $k$ and $N$
such that $Nk\leqslant T.$ 
Let $R(t):=\wU_t(t)-\nabla\cdot f(\nabla \wU(t),\wU)$ be the residual of
the reconstruction $\wU\in \V_k^{\text{c}} (q).$ Then, the following
maximal regularity a posteriori error estimate holds
\begin{equation}
\label{aposteriori-estimate1}
\|u_t - \widehat U_t\|_{L^p((0,t);L^r(\varOmega))}
+\|u - \widehat U\|_{L^p((0,t);W^{2,r}(\varOmega))}
\leqslant C\|R\|_{L^p((0,t);L^r(\varOmega))}
\end{equation}
for all $0<t\leqslant T,$
with a constant $C$ depending also on $T.$ Furthermore, the estimator
is of optimal order, i.e.,
$\|R\|_{L^p((0,T);L^r(\varOmega))}\leqslant Ck^q.$
\end{theorem}

We provide the proofs of these theorems  in Section~\ref{Se:3}
 and Section~\ref{Se:aposteriori}, respectively. 
 In Section~\ref{Se:2} we establish maximal regularity properties of dG
methods for non-autonomous linear parabolic equations
and in Section~\ref{Se:background} we recall background material.

\section{Maximal regularity results for non-autonomous linear
  equations}\label{Se:2}

In this section we establish maximal regularity properties for dG
methods for non-autonomous linear parabolic equations; these
properties play a crucial role in our error analysis for dG methods
for \eqref{ivp}.
 
Logarithmically quasi-maximal parabolic regularity results for dG
methods were established in \cite{LV1} in general Banach spaces for
autonomous equations and in \cite{LV2} in Hilbert spaces for
non-autonomous equations. These works cover also the cases of variable
time steps as well as the critical exponents $p=1,\infty.$ Maximal
regularity for Runge--Kutta and, under natural additional assumptions,
for BDF methods in UMD Banach spaces were derived in \cite{KLL}.  More
recently, another approach to the maximal regularity of dG time
discretization was presented in \cite{KK2024}. This is based on
studying the dG approximation of a temporally regularized Green's
function. The result allows quasi-uniform meshes but is restricted to
$q\geqslant 2$.

We shall use the interpretation of dG methods for linear equations as
modified Radau~IIA methods from \cite{AM-SINUM}; this will allow us to
take advantage of a maximal parabolic regularity property of Radau~IIA
methods for non-autonomous parabolic equations from \cite{KuLL}; see
Lemma~\ref{Le:maxreg-Radau} in the sequel.  In \cite{KuLL}, the
discretization of \eqref{ivp} by certain Runge--Kutta methods,
including the Radau~IIA methods, is analyzed.

\subsection{Radau~IIA methods}\label{SSe:Rad}
We first recall the Radau~IIA methods.
The $q$-stage Radau~IIA method is specified by the Radau nodes $0<c_1<\dotsb<c_q=1$ in the interval $[0,1]$, and the coefficients
\begin{equation}
\label{RK-coef}
a_{ij} =\int _0^{c_i} \ell_j (\tau )\, \d \tau, 
\quad b_{i} =\int _0^{1} \ell_i(\tau )\, \d \tau ~(= a_{qi}),
\quad   i,j=1,\dotsc,q;
\end{equation}
here, $\ell_1,\dotsc, \ell_q\in \P_{q-1}$ are the Lagrange polynomials 
for  $c_1,\dotsc,c_q,$ $\ell_i(c_j)$ $=\delta_{ij}.$
The coefficient matrix $(a_{ij})_{i,j=1,\dotsc,q}$ is invertible
and the stability function of the method vanishes at infinity.

Relations \eqref{RK-coef} reflect the fact that the Radau~IIA method is of collocation type,
i.e.,  its \emph{stage order} is $q.$  The first member of this family, for $q=1,$ is the implicit Euler method.

\emph{Auxiliary results.} We recall a discrete maximal $L^p$-regularity property for Radau~IIA methods for
non-autonomous linear parabolic equations from \cite[Lemma 3.6]{KuLL}, which will play a crucial role in our 
a priori error analysis.

Let us start with a continuous  maximal $L^p$-regularity property for non-autonomous linear parabolic equations.

\begin{lemma}[{\cite[Lemma 3.5]{KuLL}}; maximal regularity for
  non-autonomous equations]\label{Le:maxreg-nonaut}
For $0<T<\infty$ and $1<r<\infty,$ consider operators $A(t)\colon W^{2,r}(\varOmega)\cap W^{1,r}_0(\varOmega)\to L^r(\varOmega), t\in [0,T],$
of the form
\[A(t) v:=\sum_{i,j=1}^d \alpha_{ij}(\cdot,t)v_{x_ix_j},\quad  t\in [0,T];\]
without loss of generality, we assume that the coefficient functions $\alpha_{ij}(\cdot,t)\colon \varOmega\to \R, i,j=1,\dotsc,d,$
are symmetric,  $\alpha_{ij}=\alpha_{ji}.$ We also assume that the coefficients $\alpha_{ij}$ satisfy a 
uniform coercivity condition,
\begin{equation}
\label{eq:coercivity}
\sum_{i,j=1}^d \alpha_{ij}(x,t)\xi_i\xi_j\geqslant \kappa \sum_{i=1}^d |\xi_i|^2\quad \forall x\in \varOmega\  \forall t\in [0,T]\ \forall \xi\in \R^d,
\end{equation}
 with a positive constant $\kappa,$ and a Lipschitz condition in time,
\begin{equation}
\label{eq:Lipsch}
\| \alpha_{ij}(\cdot,t)-\alpha_{ij}(\cdot,s)\|_{L^\infty (\varOmega)}\leqslant L|t-s|\quad  \forall t,s\in [0,T],
\end{equation}
and are uniformly bounded in a H\"older norm, with exponent $\mu>0,$
\begin{equation}
\label{eq:Holder}
\| \alpha_{ij}(\cdot,t)\|_{C^\mu (\overline \varOmega)}\leqslant K\quad  \forall t\in [0,T].
\end{equation}
Then, the operator $A(t)$ has the maximal $L^p$-regularity property for $1<p<\infty$\emph{:}
for every $f\in  L^p\big ((0,T);L^r (\varOmega)\big ),$ with arbitrary $T>0,$
the solution $v$ of the initial and boundary value problem 
\begin{equation}
\label{ivp-maxreg}
\left \{
\begin{alignedat}{3} 
&v_t=A(t)v+f \quad && \text{in } &&\varOmega \times (0,T],\\
&v=0 \quad && \text{on } &&\partial\varOmega \times (0,T],\\
&v(0)=0\quad && \text{in } &&\varOmega,\\
\end{alignedat}
\right .
\end{equation}
is bounded in the form
\begin{equation}
\label{maxreg-nonaut}
\|v_t\|_{L^p ((0,T);L^r (\varOmega) )}+\|v\|_{L^p ((0,T);W^{2,r} (\varOmega) )}
\leqslant c_{p,r}\|f\|_{L^p ((0,T);L^r (\varOmega) )}
\end{equation}
with a constant $c_{p,r}$ depending only on the H\"older exponent $\mu$ and the corresponding bound $K$
in \eqref{eq:Holder}, on  the coercivity constant  $\kappa$  in \eqref{eq:coercivity} and the Lipschitz constant 
$L$  in \eqref{eq:Lipsch},  on  $\varOmega, p,$ and $r,$ and on the time $T.$ 
\end{lemma}

The operators
$A(t)\colon W^{2,r}(\varOmega)\cap W^{1,r}_0(\varOmega)\to
L^r(\varOmega),\ t\in [0,T],$ are invertible and
\begin{equation}
\label{eq:invertibility}
\|v \|_{W^{2,r}(\varOmega)} \leqslant  c_r\|A( t) v\|_{L^r(\varOmega)},
\end{equation}
with $c_r$ a constant depending only on $K,\kappa,\varOmega$ and $r;$
see \cite[(3.10)]{KuLL}.
Moreover, the Lipschitz condition \eqref{eq:Lipsch} implies a
  Lipschitz condition for $A$, namely, for $t,s,\bar{t}\in [0,T]$, 
  \begin{equation}
    \label{eq:lipschitzforA}
    \| (A(t)-A(s))v\|_{L^r(\varOmega)}
    \leqslant \widetilde{L} |t-s| \| A(\bar{t})v\|_{L^r(\varOmega)}.
  \end{equation}

With starting value $\V_0=0,$ we consider the discretization of the
initial and boundary value problem \eqref{ivp-maxreg} by the $q$-stage
Radau~IIA method: we recursively define approximations
$\V_\ell\in W^{2,r}(\varOmega)\cap W^{1,r}_0(\varOmega)$ to the nodal
values $v(t_\ell):=v(\cdot,t_\ell)$, as well as internal
approximations
$\V_{\ell i}\in W^{2,r}(\varOmega)\cap W^{1,r}_0(\varOmega)$ to the
intermediate values $v(t_{\ell i}),$ by
\begin{equation}\label{eq:Rad_meth}
\left \{
\begin{alignedat}{2}
&\V_{ni}=\V_n+k \sum_{j=1}^q a_{ij} [A(t_{nj})\V_{nj}+f(t_{nj})],\quad &&i=1,\dotsc,q,\\
&\V_{n+1} = \V_n+ k \sum_{i=1}^q b_i [A(t_{ni})\V_{ni}+f(t_{ni})], \quad &&
\end{alignedat}
\right . 
\end{equation}
$n=0,\dotsc,N-1$. Notice that $\V_{n+1} =\V_{nq}$, since  $c_q=1$.

An important maximal parabolic regularity property of the Radau~IIA
method \eqref{eq:Rad_meth} for initial and boundary value problems of
the form \eqref{ivp-maxreg} is:

\begin{lemma}[Maximal regularity of Radau~IIA methods;
  {\cite[Lemma~3.6]{KuLL}}]\label{Le:maxreg-Radau}
  Under the assumptions of Lemma~\ref{Le:maxreg-nonaut}, the Radau~IIA
  method \eqref{eq:Rad_meth} has the discrete maximal $L^p$-regularity
  uniformly in the step size $k$\emph{:} for
  $f(t_{ni})\in L^r(\varOmega), n=0,\dotsc,N-1, i=1,\dotsc,q,$ the
  numerical solutions
  $\V_{ni}\in W^{2,r}(\varOmega)\cap W^{1,r}_0(\varOmega)$ are well
  defined and satisfy the estimates
\begin{equation}
\label{eq:Rad-maxreg1}
\begin{aligned}
\sum_{i=1}^q\|(\partial \V_{ni})_{n=0}^M\|_{\ell^p(L^r(\varOmega))}&+\sum_{i=1}^q\|(A(\bar t)\V_{ni})_{n=0}^M\|_{\ell^p( L^r(\varOmega))}\\
&\leqslant \widetilde C_{p,r}\sum_{i=1}^q\|(  f(t_{ni}))_{n=0}^M\|_{\ell^p(L^r(\varOmega))}, 
\end{aligned}
\end{equation}
for any $\bar t \in [0,T],$ and, consequently, 
\begin{equation}
\label{eq:Rad-maxreg}
\sum_{i=1}^q\|(\partial \V_{ni})_{n=0}^M\|_{\ell^p(L^r(\varOmega))}+\sum_{i=1}^q\|(\V_{ni})_{n=0}^M\|_{\ell^p( W^{2,r}(\varOmega))}
\leqslant C_{p,r}\sum_{i=1}^q\|(  f(t_{ni}))_{n=0}^M\|_{\ell^p(L^r(\varOmega))}, 
\end{equation}
$M=0,\dotsc,N-1,$ with $\V_{-1,1}=\dotsb= \V_{-1,q}=0.$ Besides the
specific Radau~IIA method, the constants $C_{p,r}$ and
$\widetilde C_{p,r}$ depend only on the H\"older exponent $\mu$ and
the corresponding bound $K$ in \eqref{eq:Holder}, on the coercivity
constant $\kappa$ in \eqref{eq:coercivity} and the Lipschitz constant
$L$ in \eqref{eq:Lipsch}, on $\varOmega, p,r,$ and on the time $T;$ in
particular, $C_{p,r}$ is independent of the time step $k$ and of $M$
such that $Mk\leqslant T.$
\end{lemma}

Here, for a Banach space $X$ and a sequence
$(v_n)_{n\in \N_0}\subset X,$ we used the notation
\begin{equation}
  \label{eq:definediscretenorm}
  \partial v_n:=\frac {v_n-v_{n-1}}k\quad\text{and}\quad
  \|(v_n)_{n=0}^M\|_{\ell^p(X)}:=\Big (k \sum_{n=0}^M\|v_n\|_X^p\Big )^{1/p}
\end{equation}
for the backward difference quotient and for the discrete
$\ell^p(X)$-norm $\|\cdot\|_{\ell^p(X)}.$

\subsection{Discontinuous Galerkin methods}\label{SSe:dG}
We establish maximal regularity properties for dG methods for
non-autonomous linear parabolic equations.  These results are crucial for our
error analysis for dG methods for \eqref{ivp}; they are also of
independent interest.

For $n=1,\dotsc,N,$ we introduce the discrete (semi)norm
$\|\cdot \|_{\ell^p ((0,t_n);L^r(\varOmega) )}$ on
$\L_{k,0}^{\text{d}} (q-1):=\{v\in \L_k^{\text{d}} (q-1): v(0)=0\}$
and on $\L_{k,0}^{\text{c}} (q):=\{v\in \L_k^{\text{c}} (q): v(0)=0\}$
by
\begin{equation}
\label{eq:discrete-norm}
\|v \|_{\ell^p ((0,t_n);L^r(\varOmega) )}
:=\Big(\sum_{\ell=0}^{n-1}
\Big (k\sum_{i=1}^q \|v(t_{\ell i})\|_{L^r(\varOmega)}^p\Big)\Big)^{1/p}. 
\end{equation}
Since this only uses values from $[0,t_n]$, it is a norm on
  $\L_{k,0}^{\text{d}} (q-1)$ only for $t_n=T$.  The connection to 
  the standard $\ell^p(X)$-(semi)norm used in
  Lemma~\ref{Le:maxreg-Radau} and defined in
  \eqref{eq:definediscretenorm} is given by
  \begin{equation}
    \label{eq:connectionbetweennorms}
    \sum_{i=1}^q \| (v(t_{\ell i}))_{\ell=0}^{n-1} \|_{L^r(\varOmega)}^p
    =\|v \|_{\ell^p ((0,t_n);L^r(\varOmega) )}^p .
  \end{equation}

The continuous $L^p(L^r(\varOmega))$-(semi)norms and the discrete
$\ell^p(L^r(\varOmega))$-(semi)norms are equivalent, with constants
independent of $k$ and $n;$ see \cite[\textsection 2.1]{AL-linear}.
In particular, we have
\begin{equation}
\label{eq:discrete-norm-dom}
\|v \|_{L^p ((0,t_n);L^r(\varOmega) )}
\leqslant \tilde c_{q,p} \|v \|_{\ell^p ((0,t_n);L^r(\varOmega) )}
\quad \forall v\in \L_{k,0}^{\text{d}} (q-1),\ n=1,\dotsc,N.
\end{equation}
Indeed, this is proved by expressing $v$ in terms of
$\{\ell_{mi}\}_{i=1}^{q}$, the Lagrange polynomials
$ \ell_i\in \P_{q-1}$ for the Radau points $c_1,\dotsc,c_q,$ shifted
to each subinterval $ J_m$.

It is obvious from \eqref{wU_def-inter} that the discrete
$\ell^p(L^r(\varOmega))$-seminorms of $v$ and of its reconstruction
$\hat v\in \L_{k,0}^{\text{c}} (q)$ coincide,
\begin{equation}
\label{eq:discrete-norm-equal}
\|\hat v \|_{\ell^p ((0,t_n);L^r(\varOmega) )}=\|v \|_{\ell^p ((0,t_n);L^r(\varOmega) )}
\quad \forall v\in \L_{k,0}^{\text{d}} (q-1).
\end{equation}
%

We shall use the notation $\partial_k$ for the backward difference operator,
\[\partial_k\wV:=\frac {\wV (\cdot)- \wV (\cdot-k)}k,\]
for the  reconstruction $\wV$ of the discontinuous Galerkin approximation 
$V\in  \V_k^{\text{d}} (q-1)$ for  \eqref{ivp-maxreg}, 
with $\wV=0$ in the interval $[-k,0).$

Combining Lemma~\ref{Le:maxreg-Radau} with ideas and results from
\cite{AM-SINUM}, we obtain the following maximal regularity properties
for dG methods:

\begin{lemma}[Maximal regularity of dG methods in continuous norms]\label{Le:maxreg-dG}
  Under the assumptions of Lemma~\ref{Le:maxreg-nonaut}, the
  reconstruction $\wV$ of the discontinuous Galerkin approximation
  $V\in \V_k^{\text{d}} (q-1)$ for \eqref{ivp-maxreg} has the maximal
  $L^p$-regularity property, uniformly in the step size $k,$
\begin{equation}
\label{eq:dG-continuous-mr1}
\begin{aligned}
  \| \partial_k\wV\|_{\ell^p((0,T);L^r(\varOmega))}
  &+\| \wV_t\|_{L^p((0,T);L^r(\varOmega))}+\| A(\bar t)\wV\|_{L^p((0,T);L^r(\varOmega))}\\
  &+\| A(\bar t)V\|_{L^p((0,T);L^r(\varOmega))}
  \leqslant \widetilde C_{p,q,T} \|f\|_{L^p((0,T);L^r(\varOmega))},  
\end{aligned}
\end{equation}
for any $\bar t\in [0,T],$ and, consequently,
\begin{equation}
\label{eq:dG-continuous-mr2}
\begin{aligned}
  \| \partial_k\wV\|_{\ell^p((0,T);L^r(\varOmega))}
  &+\| \wV_t\|_{L^p((0,T);L^r(\varOmega))}+\| \wV\|_{L^p((0,T);W^{2,r}(\varOmega))}\\
  &+\|V\|_{L^p((0,T);W^{2,r}(\varOmega))}
  \leqslant C_{p,q,T} \|f\|_{L^p((0,T);L^r(\varOmega))},
\end{aligned}
\end{equation}
where $C_{p,q,T}$ and $\widetilde C_{p,q,T}$ are constants independent
of $k$ and $N$ such that $Nk\leqslant T.$ Notice that
$\| A(\bar t)\cdot\|_{L^p((0,T);L^r(\varOmega))}$ and
$\| A\cdot\|_{L^p((0,T);L^r(\varOmega))}$ are equivalent norms; see
Remark~\ref{Re:equiv}.
\end{lemma}

\begin{proof}
  With starting value $V(0)=V_0=0,$ the discretization of the initial
  value problem \eqref{ivp-maxreg} by the dG$(q-1)$ method is to seek
  $V\in \V_k^{\text{d}} (q-1)$ such that
\begin{equation}
\label{dg-na-sse}
\int_{J_n}   \big[( V_t ,v )  - ( A(t)V ,v ) \big] \, \d t  + ( V_n^{+}-V_n, v_n^{+})
= \int_{J_n} (f,v)\, \d t \quad \forall v \in \P(q-1)
\end{equation}
for $n=0,\dotsc,N-1$; cf.~\eqref{dg}.  %
With $\wV\in \V_k^{\text{c}} (q)$, the reconstruction of
$V\in \V_k^{\text{d}} (q-1)$ and $P_{q-1}$ the piecewise
$L^2$-projection onto $\L_k^{\text{d}}(q-1)$,
relation~\eqref{dg-na-sse} can be written as a \emph{pointwise
  equation}
\begin{equation}
\label{dg-na_pw}
\wV_t -  P_{q-1}(AV)=P_{q-1} f;
\end{equation}
cf.~\eqref{dgwU_pw}. In particular, in the case of autonomous
equations, we have $AV\in \L_k^{\text{d}}(q-1)$ and \eqref{dg-na_pw}
reduces to
\begin{equation}
\label{dg-a_pw}
\wV_t -  AV=P_{q-1} f.
\end{equation}

Let $V_n=V(t_n), V_{nj}=V(t_{n j}), j=1,\dotsc,q, $ denote the nodal
and intermediate values of the approximate solution
$V\in \V_k^{\text{d}} (q-1);$ see \eqref{dg-na-sse}.  We shall prove
in the sequel, separately for the autonomous and the non-autonomous
cases, that the dG method satisfies the \emph{discrete} maximal
regularity estimate
\begin{equation}
\label{eq:dG-discrete-mr1}
\sum_{i=1}^q\|(\partial V_{ni})_{n=0}^{N-1}\|_{\ell^p(L^r(\varOmega))}
+\sum_{i=1}^q\|(A(\bar t)V_{ni})_{n=0}^{N-1}\|_{\ell^p(L^r(\varOmega))}
\leqslant C_{p,r}\|f\|_{L^p((0,T); L^r(\varOmega))}
\end{equation}
for any $\bar t \in [0,T],$ with a constant $C_{p,r}$ independent of
$N;$ here, $V_{-1,1}=\dotsb= V_{-1,q}=0.$

Assuming \eqref{eq:dG-discrete-mr1} for the time being, we show that
it leads to the asserted \emph{continuous} maximal regularity property
\eqref{eq:dG-continuous-mr1} of the reconstruction $\wV;$ see
\cite[Theorem 1.1]{AM-SINUM}.  To prove \eqref{eq:dG-continuous-mr1},
we first notice that
\begin{equation}
\label{eq:dG-discrete-continuous1n}
\| \partial_k\wV\|_{\ell^p((0,T);L^r(\varOmega))}^p
=\sum_{i=1}^q\|(\partial V_{ni})_{n=0}^{N-1}\|_{\ell^p(L^r(\varOmega))}^p.
\end{equation}
Moreover, in view of \eqref{eq:discrete-norm-dom},
\begin{equation}
\label{eq:discrete-norm-dom2}
\| A (\bar t)V \|_{L^p ((0,T);L^r(\varOmega) )}
\leqslant \tilde c_{q,p} \|A (\bar t)V \|_{\ell^p ((0,T);L^r(\varOmega) )}.
\end{equation}
Similarly to the proof of \eqref{eq:discrete-norm-dom}, by expanding
$\wV$ in $\widehat \ell_{ni}$, the Lagrange polynomials
$\widehat \ell_i\in \P_q$ for the points $c_0=0,c_1,\dotsc,c_q=1,$
shifted to the interval $\bar J_n$, and with $ V_{n0} = V_n,$ we can
prove that
\begin{equation}
\label{eq:discrete-norm-dom3}
\| A (\bar t)\wV \|_{L^p ((0,T);L^r(\varOmega) )}\leqslant c \|A (\bar t)V \|_{\ell^p ((0,T);L^r(\varOmega) )};
\end{equation}
see also \cite[\textsection 2.3]{AM-SINUM}.

Furthermore, it is easily seen that 
\[\|P_{q-1}\varphi\|_{L^p((0,T);L^r(\varOmega))} \leqslant c_{p,q} \|\varphi\|_{L^p((0,T);L^r(\varOmega))}\]
(see  \cite[\textsection 2.3]{AM-SINUM}), whence \eqref{dg-na_pw} yields
\[\|\wV_t \|_{L^p((0,T);L^r(\varOmega))} \leqslant  c_{p,q}\big (\|A V\|_{L^p((0,T);L^r(\varOmega))}+\|f\|_{L^p((0,T);L^r(\varOmega))}\big ),\]
and thus, in view of the equivalence of the norms $\| A(\bar t)\cdot\|_{L^p((0,T);L^r(\varOmega))}$ and $\| A\cdot\|_{L^p((0,T);L^r(\varOmega))}$,
\begin{equation}
\label{eq:dG-discrete-continuous2}
\|\wV_t \|_{L^p((0,T);L^r(\varOmega))} \leqslant  \tilde c_{p,q}\big (\|A(\bar t)  V\|_{L^p((0,T);L^r(\varOmega))}+\|f\|_{L^p((0,T);L^r(\varOmega))}\big ).
\end{equation}
Now, \eqref{eq:dG-continuous-mr1} follows immediately from \eqref{eq:dG-discrete-mr1}  in view of \eqref{eq:dG-discrete-continuous1n},  
\eqref{eq:discrete-norm-dom2}, \eqref{eq:discrete-norm-dom3}, and \eqref{eq:dG-discrete-continuous2}.
The estimate \eqref{eq:dG-continuous-mr2}  follows from \eqref{eq:dG-continuous-mr1}  and \eqref{eq:invertibility}.

Let us now turn to the proof of the discrete maximal regularity
estimate \eqref{eq:dG-discrete-mr1}.  We recall from
\cite[\textsection 3.3]{AM-SINUM}, that the dG method
\eqref{dg-na-sse} can be viewed as a modified Radau~IIA method: The dG
approximations $V_n=V(t_n), V_{nj}=V(t_{n j}), j=1,\dotsc,q, $ for the
non-autonomous parabolic equation \eqref{ivp-maxreg}, given in
\eqref{dg-na-sse}, satisfy the modified Radau~IIA method
\begin{equation}\label{eq:Rad-1-dG-na-sse}
\left \{
\begin{alignedat}{2}
&V_{ni}=V_n+ k \sum_{j=1}^q a_{ij} \big (A(t_{nj}) V_{nj}  +f_{nj} + \zeta_{nj} \big ),\quad &&i=1,\dotsc,q,\\
&V_{n+1} =V_n+ k \sum_{i=1}^q b_i \big (A(t_{ni}) V_{ni}  +f_{n i}+ \zeta_{ni} \big ), \quad &&
\end{alignedat}
\right . 
\end{equation}
$n=0,\dotsc,N-1,$ with $f_{ni}$ the averages
\begin{equation}\label{average}
f_{ni}:=\frac 1 {\int _{J_n} \ell_{ni} (s) \, \d s }  \int _{J_n}  \ell_{ni} (s)f(s)  \, \d s
=\frac 1 {b_ik }  \int _{J_n}  \ell_{ni} (s) f(s)  \, \d s,\quad  i=1,\dotsc,q,
\end{equation}
and with $\zeta_{ni}$ the modifications 
\begin{equation}\label{average-A-sse}
\zeta_{ni}: =\frac 1 {b_ik } \int _{J_n}  \ell _{ni} (s)A(s) V(s)  \, \d s - A(t_{ni}) V_{ni},\quad  i=1,\dotsc,q.
\end{equation}
Here we denote by $\ell_{ni}$ the Lagrange polynomials
$ \ell_i\in \P_{q-1}$ for the Radau nodes $c_1,\dotsc,c_q,$ shifted to
the interval $J_n$. Again, $V_{n+1} =V_{nq}.$ Furthermore,
\eqref{eq:Rad-1-dG-na-sse} written in terms of the reconstruction
$\widehat V$ of $V$ is a modified collocation method in each interval
$J_n$ with starting value $\widehat V_{n0}=V_n$, namely
\begin{equation}\label{DG-collocation-na-sse}
\widehat V_t(t_{ni}) = A(t_{ni}) \widehat V(t_{ni})+ f_{ni} + \zeta _{ni}, \quad i=1,\dotsc, q.
\end{equation}

\emph{The autonomous case.}  Notice that in the case of a
time-independent operator $A,$ the terms $\zeta_{ni}$ vanish due to
the fact that the Radau quadrature formula with nodes
$t_{n1},\dotsc, t_{nq}$ integrates polynomials of degree up to $2q-2$
exactly, and \eqref{eq:Rad-1-dG-na-sse} simplifies to
\begin{equation}\label{eq:Rad-1-dG}
\left \{
\begin{alignedat}{2}
&V_{ni}=V_n+ k \sum_{j=1}^q a_{ij} \big (AV_{nj}  +f_{nj}  \big ),\quad &&i=1,\dotsc,q,\\
&V_{n+1} =V_n+ k \sum_{i=1}^q b_i \big (AV_{ni}  +f_{n i} \big ), \quad &&
\end{alignedat}
\right . 
\end{equation}
$n=0,\dotsc,N-1.$ 

First, from Lemma~\ref{Le:maxreg-Radau} and the modified Radau~IIA
formulation \eqref{eq:Rad-1-dG} of the dG method, we obtain the
preliminary stability estimate
\begin{equation*}
\label{eq:dG-prel1}
\sum_{i=1}^q\|(\partial V_{ni})_{n=0}^{N-1}\|_{\ell^p(L^r(\varOmega))}
+\sum_{i=1}^q\|(AV_{ni})_{n=0}^{N-1}\|_{\ell^p(L^r(\varOmega))}
\leqslant C_{p,r}\sum_{i=1}^q\|(  f_{ni})_{n=0}^{N-1}\|_{\ell^p(L^r(\varOmega))} 
\end{equation*}
with a constant $C_{p,r}$ independent of $N$ and $T.$ Let us
abbreviate $L^r(\varOmega)$ by $X.$ Since the term on the right-hand
side can be estimated in the form
\begin{equation}
\label{bound_f_intro}
\sum_{i=1}^q\|(  f_{ni})_{n=0}^{N-1}\|_{\ell^p(X)}
\leqslant  C \, \|f\|_{L^p((0,T);X)} 
\end{equation}
with a constant $C$ depending only on
$b_1,\dotsc,b_q,c_1\dotsc,c_q,$ and $p,$ (see
\cite[(2.12)]{AM-SINUM}), we are led to the desired discrete maximal
regularity estimate \eqref{eq:dG-discrete-mr1}.

\emph{The non-autonomous case.}  Having written the dG method for the
non-autonomous equation \eqref{ivp-maxreg} as a modified $q$-stage
Radau~IIA method, cf.~\eqref{eq:Rad-1-dG-na-sse}, we can apply
Lemma~\ref{Le:maxreg-Radau} to obtain (with $\bar t=t_m$)
\begin{equation}
\label{na-dG-1}
\begin{aligned}
  &\sum_{i=1}^q\|(\partial V_{ni})_{n=0}^{m-1}\|_{\ell^p(X)}
  +\sum_{i=1}^q\|(A(t_m)V_{ni})_{n=0}^{m-1}\|_{\ell^p(X)} \\
  &\qquad \leqslant C_{p,T}\sum_{i=1}^q\|(  f_{ni})_{n=0}^{m-1}\|_{\ell^p(X)} 
  +C_{p,T}\sum_{i=1}^q\|(  \zeta_{ni})_{n=0}^{m-1}\|_{\ell^p(X)},
\end{aligned}
\end{equation}
$m=1,\dotsc,N,$ with a constant $C_{p,T}$, independent of $m$ and $k,$
depending only on the method, i.e., on $q,$ and on
$\mu, K, \kappa, L, \varOmega, p, r,$ and $T.$ Again we use the
abbreviation $X=L^r(\varOmega)$.

We shall now show that the second term on the right-hand side of
\eqref{na-dG-1} can be absorbed in its left-hand side. This in
combination with \eqref{bound_f_intro} leads then to the desired
result.

Following the proof of the analogous result for Radau~IIA methods in
\cite{AM-NM}, we let
\[Z_m:=\sum_{i=1}^q\|(  \zeta_{ni})_{n=0}^{m-1}\|_{\ell^p(X)}^p\]
and 
\[E_\ell:=k\sum_{j=0}^{\ell-1} \sum_{i=1}^q\|A(t_m)V_{ji}\|_X^p,
  \quad \ell=1,\dotsc,m, \quad  E_0:=0.\]
According to estimate \eqref{na-dG-1}, we have
\begin{equation}\label{na-dG-3}
 E_m\leqslant C\sum_{i=1}^q\|( f_{ni})_{n=0}^{m-1}\|_{\ell^p(X)}^p+C Z_m.
\end{equation}

Since the Radau quadrature formula with nodes $t_{ni}, i=1,\dotsc,q,$
integrates polynomials of degree up to $2q-2$ exactly, we have
\[ \int _{J_n}  \ell _{ni} (s)A(t_{ni}) V(s)  \, \d s
  =  b_ik A(t_{ni}) V_{ni},\quad  i=1,\dotsc,q,\]
and can rewrite $\zeta_{ni}$ in the form
\begin{equation}\label{na-dG-4}
\zeta_{ni}=\frac 1 {b_ik }  \int _{J_n}  \ell _{ni} (s)[A(s)-A(t_{ni})] V(s)  \, \d s,\quad  i=1,\dotsc,q.
\end{equation}
Therefore, via H\"older's inequality,
\[k  \|  \zeta_{ni}\| _X ^p 
 \leqslant \gamma  \int _{J_n} \|  [A(s)-A(t_{ni})] V(s)  \|  _X^p   \, \d s   \]
with
$\gamma:=\max_{1\leqslant i \leqslant q} \Big ( \frac 1 {b_{i}^{p'} }
{\int _0^1 | \ell _{i} (\tau) |^{p'} \d\tau} \Big ) ^{p/p'}.$ By the
Lipschitz condition \eqref{eq:lipschitzforA}, we have, for $s\in J_n,$
\[\|  [A(s)-A(t_{ni})] V(s)  \|_X\leqslant \tilde{L} (t_m-t_n) \| A(t_m) V(s)  \|_X,\]
whence
\[ k  \|  \zeta_{ni}\| _X ^p 
  \leqslant \gamma  \tilde{L}(t_m-t_n)^p
  \int _{J_n}  \| A(t_m) V(s)  \|_X^p   \, \d s .\] 
By an analogue of the proof of \eqref{eq:discrete-norm-dom}, we infer that 
 \[ \int _{J_n}  \| A(t_m) V(s)  \|_X^p   \, \d s
 \leqslant c_{p,q} k \sum_{j=1}^q\|A(t_m)V_{nj}\|_X^p,\]
whence
\begin{equation}
\label{na-dG-5}
  k  \sum_{i=1}^q\|  \zeta_{ni}\| _X ^p 
 \leqslant c_{p,q}\gamma  \tilde{L} (t_m-t_n)^p  k \sum_{i=1}^q\|A(t_m)V_{ni}\|_X^p . 
\end{equation}
Therefore, applying also summation by parts, we get
\begin{align*}
  Z_m & \leqslant c_{p,q}\gamma  \tilde{L} k\sum_{\ell=0}^{m-1} (t_m-t_\ell)^p\sum_{i=1}^q\|A(t_m)V_{\ell i}\|_X^p
      =c_{p,q}\gamma  \tilde{L} \sum_{\ell=1}^m
        (t_m-t_{\ell-1})^p(E_\ell-E_{\ell-1}) \\
      &  =c_{p,q}\gamma  \tilde{L} \sum_{\ell=1}^m a_\ell E_\ell,
        \quad a_\ell = (t_m-t_{\ell-1})^p-(t_m-t_{\ell})^p.
\end{align*}
Now \eqref{na-dG-3} implies
\begin{align*}
  E_m\leqslant C\sum_{i=1}^q\|( f_{ni})_{n=0}^{m-1}\|_{\ell^p(X)}^p
  + C\sum_{\ell=1}^m a_\ell E_\ell,
\end{align*}
and, since $\sum_{\ell=0}^ma_\ell=T^p$, the discrete Gronwall lemma
implies
\begin{align*}
  E_m\leqslant C\sum_{i=1}^q\|( f_{ni})_{n=0}^{m-1}\|_{\ell^p(X)}^p.
\end{align*}
By inserting this into \eqref{na-dG-1}, we obtain the estimate
\begin{equation}
\label{na-dG-6}
\sum_{i=1}^q\|(\partial V_{ni})_{n=0}^{m-1}\|_{\ell^p(X)}+\sum_{i=1}^q\|(A(t_m)V_{ni})_{n=0}^{m-1}\|_{\ell^p(X)}
\leqslant C_{p,T}\sum_{i=1}^q\|(  f_{ni})_{n=0}^{m-1}\|_{\ell^p(X)}
\end{equation}
%
$m=1,\dotsc,N,$ with a constant $C_{p,T}$ independent of $m$ and $k.$ 

Finally, this estimate in combination with \eqref{bound_f_intro} yields the
discrete maximal parabolic regularity estimate \eqref{eq:dG-discrete-mr1}
for dG methods for non-autonomous linear equations. 
\end{proof}

\begin{remark}[Equivalence of norms $\| A(\bar t)\cdot\|_{L^p((0,T);L^r(\varOmega))}$ 
and $\| A\cdot\|_{L^p((0,T);L^r(\varOmega))}$]\label{Re:equiv}
The symmetric matrices $\AA(\cdot,t)$ with entries $\alpha_{ij}(\cdot,t)$ are invertible
according to  \eqref{eq:coercivity} and their entries are bounded from above according to \eqref{eq:Holder}.
The norms $\|\AA(x,t)\cdot \|_r$  and $\|\AA(x,t)\cdot \|_2$ (for fixed $x$ and $t$) are equivalent on $\R^d$ with constants depending only
on $r.$ Furthermore, the norms $\|\AA(x,t)\cdot\|_2$ are equivalent for all $t$ and $x,$ with constants depending only
on $\kappa$ and $K.$ These properties lead to the equivalence of the norms $\| A(\bar t)\cdot\|_{L^p((0,T);L^r(\varOmega))}$ 
and $\| A\cdot\|_{L^p((0,T);L^r(\varOmega))}$.
\end{remark}

\section{Background material}\label{Se:background}

We recall a space-time Sobolev embedding that will be crucial in the sequel.

\begin{lemma}[Space-time Sobolev embedding; {\cite[Lemma 3.1]{KuLL}}]\label{Le:Sob-embedd}
  If $2/p + d/r < 1,$ then the following embeddings hold:
\begin{align}\label{Sobolev-emb1}
&W^{1,p}(\R_+;L^r(\varOmega))\cap L^p(\R_+;W^{2,r}(\varOmega)) \hookrightarrow L^\infty(\R_+;W^{1,\infty}(\varOmega)),\\
&W^{1,p}((0,T);L^r(\varOmega))\cap L^p((0,T);W^{2,r}(\varOmega)) \hookrightarrow L^\infty((0,T);W^{1,\infty}(\varOmega)) 
\label{Sobolev-emb2}
\end{align}
and the second embedding is compact.
\end{lemma}

A consequence of the embedding \eqref{Sobolev-emb2} is the following
\emph{space-time Sobolev inequality}: If $2/p + d/r < 1$ and
$v\in W^{1,p}((0,T);L^r(\varOmega))\cap L^p((0,T);W^{2,r}(\varOmega))
$ such that $v(0)=0,$ then
\begin{equation}
\label{Sobolev-ineq}
\|v\|_{L^\infty ((0,T);W^{1,\infty}(\varOmega) )}\leqslant c_{p,r} \big (\|v_t\|_{L^p ((0,T);L^r (\varOmega)  )}
+\|v\|_{L^p ((0,T);W^{2,r} (\varOmega)  )}\big );  
\end{equation}
see \cite[(3.1)]{KuLL}. 

We also recall Schaefer's fixed point theorem.

\begin{lemma}[Schaefer’s fixed point theorem; {\cite[Chapter 9.2,
    Theorem 4]{E}}] \label{Le:Schaefer} 
  Let $X$ be a Banach space and let $\M\colon X \to X$ be a continuous
  and compact map. If the set
\begin{equation}
\label{eq:Sch-fp}
  \{\varphi\in X: \varphi=\mu \M(\varphi) \text{ for some }
    \mu\in [0,1]\} 
\end{equation}
is bounded in $X$, then the map $\M$ has a fixed point.
\end{lemma}

\section{A priori error analysis}\label{Se:3} 

As in \cite[\textsection 4]{KuLL}, for notational convenience, we
shall prove the a priori error estimates, i.e.,
Theorem~\ref{Theorem1}, for the case that $f$ depends only on
$\nabla u,$ but not explicitly on $u$. The general case results in
lower order terms that are nonessential for maximal regularity
properties and cause no additional difficulties but make the notation
more cumbersome.  In the error analysis, we shall rewrite the
differential equation $u_t=\nabla\cdot f(\nabla u)$ in the form
\begin{equation}
\label{eq:diff-eq2n}
u_t=A(\nabla u)u ,
\end{equation}
where the second order operator
$A(\nabla v)\colon W^{2,r}(\varOmega)\cap W^{1,r}_0(\varOmega)\to
L^r(\varOmega)$, for fixed $v\in W^{1,\infty}(\varOmega) $, is
defined by
\begin{equation}
\label{eq:notation-A}
A(\nabla v)w:=\sum_{\ell,m=1}^d \alpha_{\ell m}(\nabla v)w_{x_\ell x_m}, 
\end{equation}
where
\begin{equation}
\label{eq:notation-f_lm}
\alpha_{\ell m}(p):=[f_{\ell m}(p)+f_{m \ell}(p)]/2, \quad
f_{\ell m}(p):=\frac {\partial f_\ell}{\partial p_m}(p) \quad \forall p\in \R^d;
\end{equation}
cf., also, \cite[\textsection 4]{KuLL}.  The form \eqref{eq:diff-eq2n}
of the nonlinear parabolic equation will allow us to take advantage of
the maximal regularity property of dG methods for non-autonomous linear
parabolic equations (cf.\ Lemma~\ref{Le:maxreg-dG}) and to write the
derivation of the estimates in compact form; for instance, we shall
use the local Lipschitz condition for $\alpha_{\ell m}$ to estimate
differences of the form $\big [A(\nabla v_1)-A(\nabla v_2)\big ]w$ for
bounded $\nabla v_1$ and $\nabla v_2$.  More precisely: for
$v_1,v_2\in W^{1,\infty}(\varOmega)$ such that
\begin{equation} \label{eq:M-bound}
\|\nabla v_1\|_{L^{\infty}(\varOmega)}\leqslant M, \ \
  \|\nabla v_2\|_{L^{\infty}(\varOmega)}\leqslant M,
\end{equation} 
the local Lipschitz conditions for
$\alpha_{\ell m}\colon \R^d\to\R, \ell,m=1,\dotsc,d,$ (see
\eqref{eq:notation-A} and \eqref{eq:notation-f_lm}), lead to a local
Lipschitz condition for the operator $\nabla v\mapsto A(\nabla v),$
namely
\begin{equation}
\label{Lipschitz-A}
\big\|\big(A(\nabla v_1)- A(\nabla v_2) \big)w\big\|_{L^r(\varOmega)}
  \leqslant  L_M \| \nabla v_1- \nabla v_2 \|_{L^{\infty}(\varOmega)}
  \|w\|_{W^{2,r}(\varOmega)}.
\end{equation}

\subsection{An interpolant and its approximation properties}\label{SSe:3.1}
We shall use a standard interpolant
$\tilde u\in \V_k^{\text{d}} (q-1)$ of the solution $u$ such that
$\tilde u(t_n)=u(t_n), n=0,\dotsc,N,$ and $u-\tilde u$ is in each
subinterval $J_n$ orthogonal to polynomials of degree at most $q-2$
(with the second condition being void for $q=1$).  Thus, $\tilde u$ is
determined in $J_n$ by the conditions
\begin{equation}
\label{eq:tilde-u1}
\left\{
\begin{aligned}
&\tilde u(t_{n+1})=u(t_{n+1}),\\
&\int_{t_n}^{t_{n+1}} \big (u(t)-\tilde u(t)\big )t^j\, \d t=0,\quad j=0,\dotsc,q-2;
\end{aligned}
\right.
\end{equation}
cf., e.g., \cite[(12.9)]{T} and \cite[\textsection\textsection
3.1--3.3]{SS} for the case of Hilbert spaces.

Existence and uniqueness of $\tilde u$ as well as approximation
properties for the interpolant $\tilde u$ and its reconstruction
$\hat{\tilde u}$ are established in \cite{AL-linear} for general
Banach spaces $X.$ The following approximation properties
\begin{align}
\label{eq:approx-prop-desired1}
\|u-\tilde u\|_{L^p((0,T);X)} &\leqslant ck^q\|u^{(q)}\|_{L^p((0,T);X)},\\
\label{eq:approx-prop-desired1B-infty}
\|u-\tilde u\|_{L^\infty((0,T);X)} &\leqslant
  ck^{q-\frac1p}\|u^{(q)}\|_{L^p((0,T);X)}, \\
\label{eq:approx-prop-desired2}
\|u-\hat {\tilde u}\|_{L^p((0,T);X)}&\leqslant c k^{q}\|u^{(q)}\|_{L^p((0,T);X)},\\  
\label{eq:approx-prop-desired3}
\|u_t-\hat {\tilde u}_t\|_{L^p((0,T);X)}&\leqslant c k^{q}\|u^{(q+1)}\|_{L^p((0,T);X)},  
\end{align}
with $X=L^r(\varOmega)$, $X= W^{1,\infty}(\varOmega),$ or
$X= W^{2,r}(\varOmega),$ will play an important role in our analysis;
see \cite[(2.13), (A.8), (2.14), (2.15)]{AL-linear}.  Here, we used the
notation $u(t):=u(\cdot,t)$ and
$u^{(q)}(t):=\frac {\partial^q u}{\partial t^q}(\cdot,t).$

The norms $\|\cdot\|_{W^{2,r}(\varOmega)}$ and
$\|A(\nabla u)\cdot\|_{L^r(\varOmega)}$ are obviously equivalent on
$W^{2,r}(\varOmega)\cap W^{1,r}_0(\varOmega)$ uniformly with respect
to $u$; see also \eqref{eq:invertibility}.  Therefore,
\eqref{eq:approx-prop-desired1} for $X=W^{2,r}(\varOmega)$ leads to
the estimate
\begin{equation}
\label{eq:estim-rho}
\|A(\nabla u)(u-\tilde u)\|_{L^p((0,T);L^r(\varOmega))}
\leqslant ck^q\|u^{(q)}\|_{L^p((0,T);W^{2,r}(\varOmega))}.
\end{equation}

Our assumption $2/p+d/r<1$ implies that $r>d$, so that we may take
advantage of the Sobolev embedding
$W^{2,r}(\varOmega) \hookrightarrow W^{1,\infty}(\varOmega);$ see
\cite[Theorem 4.12, (2)]{AF-2003}.  Then,
\eqref{eq:approx-prop-desired1B-infty} for $X=W^{2,r}(\varOmega)$
yields also the estimate
\begin{equation}
\label{eq:estim-rho-W1}
\|u-\tilde u\|_{L^\infty((0,T);W^{1,\infty}(\varOmega))}
\leqslant ck^{q-1/p} \|u^{(q)}\|_{L^p((0,T);W^{2,r}(\varOmega))}. 
\end{equation}

We use the standard splitting of the error
\begin{equation}
  \label{eq:errorsplitting}
  e=u-U=(u-\tilde u)+(\tilde u-U)=\rho+\vartheta,
\end{equation}
where we thus have various error estimates for $\rho:=u-\tilde u$ and
it remains to bound $\vartheta:=\tilde u-U\in \V_k^{\text{d}} (q-1)$.

Since $\rho=u-\tilde u$ is orthogonal to $v_t$ in $J_n$ for any
$v\in \V_k^{\text{d}} (q-1)$ and vanishes at the nodes $t_m$,
integration by parts shows that it has the following crucial property
\begin{equation}
\label{eq:rho1}
\int_{J_n}   ( \rho_t ,v )  \, \d t
+ ( \rho_n^{+}-\rho_n, v_n^{+}) =0\quad \forall v \in \P(q-1);
\end{equation}
cf.\ \cite[(12.13)]{T}.

\subsection{Regularity assumptions} \label{SSe:regularity-assumptions}

Since we assume that \eqref{eq:diff-eq2n} has a solution that belongs to
$W^{q,p}\big ((0,T);W^{2,r}(\varOmega)\cap W^{1,r}_0(\varOmega)\big)$,
it is possible to choose a single constant $M$ such that the following
bounds hold for $u$ and $\rho=u-\tilde u$.  First, the main regularity
bound, 
\begin{equation}
  \label{eq:5}
  \|u\|_{W^{q,p}((0,T);W^{2,r}(\varOmega))}
  \leqslant {M},
\end{equation}
and, by the Sobolev embedding, 
\begin{equation}
  \label{eq:7a}
  \|u\|_{L^\infty((0,T);W^{1,\infty}(\varOmega))}
  \leqslant 
  C \|u\|_{W^{q,p}((0,T);W^{2,r}(\varOmega))}
  \leqslant M. 
\end{equation}
Then, by \eqref{eq:approx-prop-desired1} and \eqref{eq:approx-prop-desired1B-infty},
\begin{align}
  \label{eq:6b}
  \|\rho\|_{L^p((0,T);W^{1,\infty}(\varOmega))}
  &\leqslant
    C \|\rho\|_{L^p((0,T);W^{2,r}(\varOmega))}
    \leqslant  Mk^{q}, 
  \\
  \label{eq:6a}
  \|\rho\|_{L^\infty((0,T);W^{1,\infty}(\varOmega))}
  &\leqslant C\|\rho\|_{L^\infty((0,T);W^{2,r}(\varOmega))}
    \leqslant Mk^{q-1/p}.
\end{align}
We also have
\begin{equation}
  \label{eq:8}
  \|\tilde{u}\|_{L^\infty((0,T);W^{2,r}(\varOmega))}
  \leqslant 
  C \|u\|_{L^\infty((0,T);W^{2,r}(\varOmega))}
  \leqslant 
  C \|u\|_{W^{q,p}((0,T);W^{2,r}(\varOmega))}
  \leqslant M. 
\end{equation}

\subsection{Derivation of the error equation} \label{SSe:derivation-err-eq}

Subtracting \eqref{dg}, with $f$ depending only on $\nabla U$, from
\begin{equation}
\label{dg-u1}
\int_{J_n}   \big( ( u_t ,v )  + (f(\nabla u) ,\nabla v ) \big) \, \d t  
+ ( u_n^{+}-u_n, v_n^{+})= 0 \quad \forall v \in \P(q-1),
\end{equation}
where $u_n^{+}-u_n=0$, we get the Galerkin orthogonality for the error
$e:=u-U$,
\begin{equation}
  \label{eq:dg-orthogonality}
  \int_{J_n}   \big[ ( e_t ,v )
  + (f(\nabla u)- f(\nabla U),\nabla v ) \big] \, \d t  
  + ( e_n^{+}-e_n, v_n^{+})=0  \quad \forall v \in \P(q-1). 
\end{equation}
We linearize  \eqref{eq:dg-orthogonality} about $\nabla u$:
\begin{equation} \label{eq:linearization}
  \nabla\cdot f(\nabla u)-\nabla\cdot f(\nabla U)
  =A(\nabla u)(u-U) 
  + \big(A(\nabla u)-A(\nabla U)\big)U 
  =A(\nabla u)e +G
\end{equation}
with remainder (recall that $\vartheta=\tilde u-U$)
\begin{equation}
\label{eq:G1} 
G:=  \big(A(\nabla u)-A(\nabla U)\big)U
=\big [A(\nabla u)-A(\nabla (\tilde u-\vartheta))\big ](\tilde u-\vartheta)
\end{equation} 
to get the linearized error equation
\begin{equation}
  \label{eq:dg-orthogonality2}
\int_{J_n}   \big[ ( e_t ,v )  - (A(\nabla u)e, v ) \big] \, \d t  
+ ( e_n^{+}-e_n, v_n^{+})
=\int_{J_n}   (G,v ) \, \d t  
\quad \forall v \in \P(q-1). 
\end{equation}
Thus, in view of \eqref{eq:errorsplitting} and \eqref{eq:rho1}, we
obtain the following equation for $\vartheta\in \V_k^{\text{d}}(q-1)$:
\begin{equation}
  \label{eq:error.dg-thetaequation4nx}
  \begin{aligned}
    & \int_{J_n}   \big[ ( \vartheta_t ,v )
    - (A(\nabla u)\vartheta , v ) \big] \, \d t
    + ( \vartheta_n^{+}-\vartheta_n, v_n^{+})\\
    & \qquad = \int_{J_n} \big [
    \big(A(\nabla u)\rho , v \big) +
    \big((A(\nabla u)-A(\nabla\tilde{u}-\nabla\vartheta))
    ( \tilde{u}-\vartheta),v \big)\big ]\, \d t  
  \end{aligned}
\end{equation}
for $v \in \P(q-1),$ for $n=0,\dotsc,N-1$, with vanishing starting
value $\vartheta(0)=0$.  This is a nonlinear equation for
$\vartheta$. In the next subsection we show the existence of a
solution with
$\| \vartheta \|_{L^\infty((0,T);W^{1,\infty}(\varOmega))}\leqslant \nu$
for arbitrarily small $\nu$.

\subsection{Existence}\label{SSe:existenceN}

Our goal here is to prove existence of a solution $\vartheta$
of \eqref{eq:error.dg-thetaequation4nx} in the space
$X=L^{\infty}((0,T),W^{1,\infty}(\varOmega))$.  We will truncate the
nonlinearity and then apply Schaefer's fixed point theorem. 

\subsubsection{A truncated equation} \label{sss.step0}

We introduce a truncation of the nonlinear operator
$\varphi\mapsto A(\nabla \tilde{u} -\nabla{\varphi})$ that appears in
the error equation \eqref{eq:error.dg-thetaequation4nx}.  We replace
it by
$\varphi\mapsto A(\nabla \tilde{u} - \beta(\varphi)\nabla{\varphi})$,
where the functional $\beta\colon W^{1,\infty}(\varOmega)\to \R$ is
defined by
\begin{equation}
  \label{eq:definitionbeta1N}
  \beta(\varphi) 
  =\min\Big(
  \frac{{\varepsilon}}{\|\varphi\|_{W^{1,\infty}(\varOmega)}},1
  \Big)\quad \text{with $\varepsilon\in(0,1]$}.
\end{equation}
Then, we have 
\begin{align}
  \label{eq:definitionbeta2cN}
  &\|\beta(\varphi_1)\nabla\varphi_1-\beta(\varphi_2)\nabla\varphi_2\|_{L^{\infty}(\varOmega)}
    \leqslant 2 \|\nabla(\varphi_1-\varphi_2)\|_{L^{\infty}(\varOmega)}, \\ 
  \label{eq:definitionbeta2aN}
  &\|\beta(\varphi)\nabla\varphi\|_{L^{\infty}(\varOmega)}
    \leqslant {\varepsilon}, \\ 
  \label{eq:definitionbeta2bN}
  &\beta(\varphi)=1, \quad \text{if
    $\|\varphi\|_{W^{1,\infty}(\varOmega)}
    \leqslant{\varepsilon}$}.
\end{align}
The proof of \eqref{eq:definitionbeta2cN} is obtained by noting that 
\begin{equation*}
  \beta(\varphi)\nabla\varphi
  =\left\{
  \begin{alignedat}{2}
  &\nabla\varphi, &&\text{if }\,  \|\varphi\|_{W^{1,\infty}(\varOmega)}\leqslant\varepsilon,\\
   & \dfrac{\varepsilon}{\|\varphi\|_{W^{1,\infty}(\varOmega)}}\nabla\varphi,\quad 
    &&\text{if }\, \|\varphi\|_{W^{1,\infty}(\varOmega)}\geqslant\varepsilon,
  \end{alignedat}
  \right.
\end{equation*}
and considering the three cases whether $\varphi_1$, $\varphi_2$ are
both small, both big, and one is big and one is small relative to
$\varepsilon$.

We thus introduce the truncated error equation to find
$\vartheta\in\V_k^{\text{d}} (q-1),$ with $\vartheta(0)=0,$ such that
\begin{equation}
  \label{eq:Appendix.dg-thetaequation4nnNyy}
  \begin{aligned}
    &
    \int_{J_n}   \big[
    ( \widehat{\vartheta}_t ,v )
    - (A(\nabla u)\vartheta , v )
    \big] \, \d t
    \\ & \qquad
    =
    \int_{J_n} \big[
    \big( A(\nabla u)\rho , v \big)
    +
    \big(
    (A(\nabla u)
    -
    A(\nabla \tilde u-\beta(\vartheta)\nabla\vartheta)
    )
    (\tilde{u}-\vartheta),v \big)
    \big] \, \d t  
  \end{aligned}
\end{equation}
for $v \in \P(q-1),$ for $n=0,\dotsc,N-1$.  Notice, that we have used
the same notation $\vartheta$ for the solutions of
\eqref{eq:Appendix.dg-thetaequation4nnNyy} and
\eqref{eq:error.dg-thetaequation4nx}.

Before proceeding, we derive a useful inequality for the nonlinear
term. By \eqref{eq:7a} we have
\begin{align}
  \label{eq:uniform-M-bounds1}
  \|\nabla u \|_{L^{\infty}((0,T);L^{\infty}(\varOmega))}
  \leqslant M
\end{align}
and, by \eqref{eq:8} and \eqref{eq:definitionbeta2aN},
\begin{align}\label{eq:uniform-M-bounds2}
  \begin{aligned}
    \|\nabla \tilde{u}-\beta(\varphi)\nabla\varphi
    \|_{L^{\infty}((0,T);L^{\infty}(\varOmega))}
    & \leqslant
    \|\nabla \tilde{u}\|_{L^{\infty}((0,T);L^{\infty}(\varOmega))}
    +\|\beta(\varphi)\nabla\varphi \|_{L^{\infty}((0,T);L^{\infty}(\varOmega))}
    \\ &
    \leqslant M+\varepsilon 
    \leqslant M+1 , 
  \end{aligned}
\end{align}
so that the local Lipschitz bound in \eqref{Lipschitz-A} with constant
$L=L_{M+1}$, \eqref{eq:6a}, and \eqref{eq:definitionbeta2aN} give
\begin{equation}
 \label{eq:uniform-lipschitz4xx}
  \begin{aligned}
  & \|(A(\nabla u)-A(\nabla \tilde{u}-\beta(\varphi)\nabla\varphi))
  v\|_{L^p((0,T);L^{r}(\varOmega))}
  \\ & \quad \leqslant 
  L \big(
  \| \nabla(u-\tilde{u}) \|_{L^{\infty}((0,T);L^{\infty}(\varOmega))}
  + \| \beta(\varphi)\nabla\varphi)
  \|_{L^{\infty}((0,T);L^{\infty}(\varOmega))}
  \big) 
  \| v \|_{L^p((0,T);W^{2,r}(\varOmega))}
  \\ & \quad \leqslant
  L\big( Mk^{q-1/p}+\varepsilon\big) 
  \| v \|_{L^p((0,T);W^{2,r}(\varOmega))}.
  \end{aligned} 
\end{equation}

\subsubsection{A fixed point equation} \label{sss.step1}

To write equation \eqref{eq:Appendix.dg-thetaequation4nnNyy} in a fixed-point
form, with a given input $\varphi \in X,$ we consider the linear
equation to find $\vartheta\in\V_k^{\text{d}} (q-1),$ with
$\vartheta(0)=0,$ such that
\begin{equation}
  \label{eq:Appendix.dg-thetaequation4nnN}
  \begin{aligned}
    &
    \int_{J_n}   \big[
    ( \widehat{\vartheta}_t ,v )
    - (A(\nabla u)\vartheta , v )
    \big] \, \d t
    \\ & \qquad
    =
    \int_{J_n} \big[
    \big( A(\nabla u)\rho , v \big)
    +
    \big(
    (A(\nabla u)
    -
    A(\nabla \tilde u-\beta(\varphi)\nabla\varphi)
    )
    (\tilde{u}-\vartheta),v \big)
    \big] \, \d t  
  \end{aligned}
\end{equation}
for $v \in \P(q-1),$ for $n=0,\dotsc,N-1$. Notice, that we have used
the same notation $\vartheta$ for the solutions of
\eqref{eq:Appendix.dg-thetaequation4nnN} and
\eqref{eq:error.dg-thetaequation4nx}.

Our next goal is to show that \eqref{eq:Appendix.dg-thetaequation4nnN}
defines a nonlinear map $\mathcal{M}\colon X\to X$,
$\varphi\mapsto\vartheta$, with
$X=L^{\infty}((0,T);W^{1,\infty}(\varOmega)) $.  It suffices to show
that, for each $\varphi\in X,$
\eqref{eq:Appendix.dg-thetaequation4nnN} possesses a unique solution
$\vartheta\in \V_k^{\text{d}} (q-1)$.

We apply the Banach fixed-point theorem to the operator
$\mathcal{S}\colon \psi\mapsto \vartheta $ on $\V_k^{\text{d}} (q-1)$,
where $\vartheta$ solves $\vartheta(0)=0$ and
\begin{equation}
  \label{eq:Appendix.dg-thetaequation4nnNN}
  \begin{aligned}
    &
    \int_{J_n}   \big[
    ( \widehat{\vartheta}_t ,v )
    - (A(\nabla u)\vartheta , v )
    \big] \, \d t
    \\ & \qquad
    =
    \int_{J_n} \big[
    \big( A(\nabla u)\rho , v \big)
    +
    \big(
    (A(\nabla u)
    -
    A(\nabla \tilde u-\beta(\varphi)\nabla\varphi)
    )
    (\tilde{u}-\psi),v \big)
    \big] \, \d t  
  \end{aligned}
\end{equation}
for $v \in \P(q-1),$ for $n=0,\dotsc,N-1$. This is a linear dG
equation and it clearly has a unique solution
$\vartheta\in\V_k^{\text{d}} (q-1)$ for each
$\psi\in\V_k^{\text{d}} (q-1)$; thus $\mathcal{S}$ is well defined.

We want to show that $\mathcal{S}$ is a contraction on
$\V_k^{\text{d}} (q-1)$.  For
$\psi_1,\psi_2\in \V_k^{\text{d}} (q-1),$ by subtracting the
corresponding equations, we see that the difference 
$\zeta=\mathcal{S}(\psi_1)-\mathcal{S}(\psi_2)$ satisfies
\begin{equation}%
  \label{eq:dg-thetaequation6N}
  \begin{aligned}
   & \int_{J_n}   \big[ ( \widehat{\zeta}_t ,v )
    - (A(\nabla u)\zeta , v )
    \big] \, \d t \\ 
    & \qquad
    = -\int_{J_n} \big((A(\nabla u)-A(\nabla \tilde{u}-\beta(\varphi)\nabla\varphi))
   (\psi_1-\psi_2),v \big)  \, \d t  
  \end{aligned}
\end{equation}
for $v \in \P(q-1),$ for $n=0,\dotsc,N-1$.

The maximal regularity property \eqref{eq:dG-continuous-mr2} of the dG
method applied to \eqref{eq:dg-thetaequation6N} yields
\begin{equation}
\label{eq:dG-Banach}
\|\zeta\|_{L^p((0,T);W^{2,r}(\varOmega))}
\leqslant C_{p,q,T}  \|(A(\nabla u)-A(\nabla \tilde{u}-\beta(\varphi)\nabla\varphi))
   (\psi_1-\psi_2)\|_{L^p((0,T);L^{r}(\varOmega))}.
\end{equation}
Here \eqref{eq:uniform-lipschitz4xx} gives 
\begin{equation*}
  \begin{aligned}
  & \|(A(\nabla u)-A(\nabla \tilde{u}-\beta(\varphi)\nabla\varphi))
  (\psi_1-\psi_2)\|_{L^p((0,T);L^{r}(\varOmega))}
  \\ & \quad
  \leqslant
  L\big( Mk^{q-1/p}+\varepsilon\big) 
  \| \psi_1-\psi_2 \|_{L^p((0,T);W^{2,r}(\varOmega))}.
  \end{aligned}
\end{equation*}
Hence,
\begin{equation}
\label{eq:dG-Banach2}
\|\mathcal{S}(\psi_1)-\mathcal{S}(\psi_2)\|_{L^p((0,T);W^{2,r}(\varOmega))}
\leqslant C_{p,q,T}  
  L \big( Mk^{q-1/p}+\varepsilon\big) 
  \| \psi_1-\psi_2 \|_{L^p((0,T);W^{2,r}(\varOmega))}.
\end{equation}
If $C_{p,q,T} L ( Mk^{q-1/p}+\varepsilon) \leqslant 1/2$, then  we
have a contraction. Thus, for every $\varphi\in X$ we have a unique
fixed point $\vartheta=\mathcal{S}(\vartheta)$.  This defines the
nonlinear mapping $\mathcal{M}\colon X\to X$ with
$\mathcal{M}(\varphi):=\vartheta$.

\subsubsection{Application of Schaefer's fixed point
  theorem} \label{sss.step2}

We check the assumptions of Theorem~\ref{Le:Schaefer} for the mapping
$\mathcal{M}\colon X\to X$.

\textit{1.  The set
  $\{\varphi\in X: \varphi=\mu \M(\varphi) \text{ for some } \mu\in
  [0,1]\} $ is bounded in $X$. }
Suppose thus that $\varphi \in X $
satisfies $\varphi=\mu\mathcal{M}(\varphi)$ for some $\mu\in[0,1]$.
This means that $\varphi=\mu\vartheta$, where
$\vartheta\in \V_k^{\text{d}} (q-1)$ is the solution of
\eqref{eq:Appendix.dg-thetaequation4nnN} with input
$\varphi=\mu\vartheta$, that is, $\vartheta(0)=0$ and
\begin{equation}
  \label{eq:dg-thetaequation4b1}
  \begin{aligned}
    &
    \int_{J_n}   \big[
    ( \widehat{\vartheta}_t ,v )
    -( A(\nabla u)\vartheta ,v)
    \big] \, \d t
    \\ & \qquad
    =
    \int_{J_n} \big[
    ( A(\nabla u)\rho , v )
    +
    \big(
    (A(\nabla u)
    -
    A(\nabla \tilde u-\beta(\mu\vartheta)\nabla (\mu\vartheta)) 
    )
    (\tilde{u}-\vartheta),v \big)
    \big] \, \d t  
  \end{aligned}
\end{equation}
for $v \in \P(q-1)$, $n=0,\dotsc,N-1$.  Maximal regularity (see
\eqref{eq:dG-continuous-mr2}) gives
\begin{equation}
\label{eq:derror-equation3aNx}
\begin{aligned}
  &
  \|\widehat\vartheta_t\|_{L^p((0,T);L^r(\varOmega))}
  +
  \|\widehat\vartheta\|_{L^p((0,T);W^{2,r}(\varOmega))}
  + \|\vartheta\|_{L^p((0,T);W^{2,r}(\varOmega))}
  \\ &\quad 
  \leqslant
  C_{p,q,T}
  \big(
  \|A(\nabla u)\rho\|_{L^p((0,T);L^r(\varOmega))}
  \\ &\qquad 
  +
  \|
  (A(\nabla u)-A(\nabla \tilde u-\beta(\mu\vartheta)\nabla(\mu\vartheta)))
  (\tilde{u}-\vartheta)
  \|_{L^p((0,T);L^r(\varOmega))}\big),
\end{aligned}
\end{equation}
where, in view of \eqref{eq:6b},
\begin{equation}
  \label{eq:9Nx}
  \|A(\nabla u)\rho\|_{L^p((0,T);L^r(\varOmega))}\big ) 
  \leqslant Mk^q 
\end{equation}
and, in view of \eqref{eq:uniform-lipschitz4xx} and \eqref{eq:6a},
\eqref{eq:8}, \eqref{eq:definitionbeta2aN},
\begin{equation}
  \label{eq:10Nx}
  \begin{aligned}
    &
    \|
    (A(\nabla u)-A(\nabla \tilde u-\beta(\mu\vartheta)\nabla(\mu\vartheta)))
    (\tilde{u}-\vartheta)
    \|_{L^p((0,T);L^r(\varOmega))}
    \\ & \quad 
    \leqslant 
    L
    \big(
    Mk^{q-1/p}+\varepsilon
    \big)
    \big(
    \| \tilde{u} \|_{L^p((0,T);W^{2,r}(\varOmega))}  
    +
    \| \vartheta \|_{L^p((0,T);W^{2,r}(\varOmega))}  
    \big)
    \\ & \quad 
    \leqslant 
    L
    \big(
    Mk^{q-1/p}+\varepsilon
    \big)
    \big(M+\| \vartheta \|_{L^p((0,T);W^{2,r}(\varOmega))} \big) .
  \end{aligned}
\end{equation}
Hence, 
\begin{equation}
\label{eq:derror-equation3aN2x}
\begin{aligned}
  &
  \|\widehat\vartheta_t\|_{L^p((0,T);L^r(\varOmega))}
  +
  \|\widehat\vartheta\|_{L^p((0,T);W^{2,r}(\varOmega))}
  + \|\vartheta\|_{L^p((0,T);W^{2,r}(\varOmega))}
  \\ &\quad 
  \leqslant
  C_{p,q,T}
  \big(
  Mk^q 
  +
  LM
  \big(
  M k^{q-1/p}+\varepsilon
  \big)
  \big)
  +
  C_{p,q,T} L
  \big(
  M k^{q-1/p}+\varepsilon
  \big)
  \| \vartheta \|_{L^p((0,T);W^{2,r}(\varOmega))}  .
 \end{aligned}
\end{equation}
If $C_{p,q,T} L ( Mk^{q-1/p}+\varepsilon)\leqslant 1/2$ (the
same smallness condition as before), then the last term can be hidden
in the left-hand side to get
\begin{equation}
\label{eq:derror-equation3aN3x}
\begin{aligned}
  &
  \|\widehat\vartheta_t\|_{L^p((0,T);L^r(\varOmega))}
  +
  \|\widehat\vartheta\|_{L^p((0,T);W^{2,r}(\varOmega))}
  + \|\vartheta\|_{L^p((0,T);W^{2,r}(\varOmega))}
  \\ & \qquad
  \leqslant
  C 
  \big(
  Mk^q 
  +
  LM
  \big(
  M k^{q-1/p}+\varepsilon
  \big)
  \big).  
\end{aligned}
\end{equation}
Hence, by the space-time Sobolev inequality
\eqref{Sobolev-ineq}, we also have
\begin{equation}
\label{Sobolev-ineq2x}
\begin{aligned}
  &\| \widehat{\vartheta} \|_{L^\infty ((0,T);W^{1,\infty}(\varOmega) )}
  \leqslant c_{p,r} \big (\|\widehat{\vartheta}_t\|_{L^p ((0,T);L^r (\varOmega)  )}
  +\|\widehat{\vartheta} \|_{L^p ((0,T);W^{2,r} (\varOmega)  )}\big )  
  \\ & \qquad
  \leqslant
  C 
  \big(
  Mk^q 
  +
  LM
  \big(
  M k^{q-1/p}+\varepsilon
  \big)
  \big). 
\end{aligned}
\end{equation}
Since $\vartheta$ is a (discontinuous) Lagrange interpolant of
$\widehat{\vartheta}$, it follows that
\begin{equation}
\label{Sobolev-ineq2x2}
\| \vartheta \|_{L^\infty ((0,T);W^{1,\infty}(\varOmega) )}
\leqslant
C Mk^q+C LM(Mk^{q-1/p}+\varepsilon) .  
\end{equation}
Note that $C$ does not depend on $\mu$, $k$ and $\varepsilon$.  This
shows that the set
$\{\varphi\in X: \varphi=\mu \M(\varphi) \text{ for some } \mu\in
[0,1]\} $ is bounded in $X$.
 
\textit{2.  $\mathcal{M}\colon X\to X$ is compact.}  The previous
calculation with $\mu=1 $ shows that
$\varphi\mapsto \widehat{\vartheta}$ maps
$L^\infty((0,T);W^{1,\infty}(\varOmega)) $ into
$W^{1,p}((0,T);L^r(\varOmega))\cap L^p((0,T);W^{2,r}(\varOmega))$,
which is compactly embedded in
$L^\infty((0,T);W^{1,\infty}(\varOmega))$; see \eqref{Sobolev-emb2}.
Since $\vartheta$ is a (discontinuous) Lagrange interpolant of
$\widehat{\vartheta}$, it follows that $\mathcal{M}$ maps
$X=L^{\infty}((0,T),W^{1,\infty}(\varOmega))$ compactly into itself.

\textit{3.  $\mathcal{M}\colon X\to X$ is continuous.}
The function $\vartheta=\mathcal{M}(\varphi)$ is the unique solution
of \eqref{eq:Appendix.dg-thetaequation4nnN} and therefore the
difference
$\zeta=\vartheta_1-\vartheta_2=\mathcal{M}(\varphi_1)-\mathcal{M}(\varphi_2)$
satisfies $\zeta(0)=0$ and
\begin{equation}
  \label{eq:dg-thetaequation6N2}
    \int_{J_n}   \big[ ( \widehat{\zeta}_t ,v )
    - (A(\nabla u)\zeta , v )
    \big] \, \d t
    = \int_{J_n} ( G,v)  \, \d t  
\end{equation}
for $v \in \P(q-1),$ for $n=0,\dotsc,N-1$, where
\begin{equation}
  \label{eq:dg-thetaequation6N2b}
  \begin{aligned}
    G
    ={}&
    (A(\nabla u) 
    -
    A(\nabla \tilde{u}-\beta(\varphi_1)\nabla\varphi_1))
    (\tilde{u}-\vartheta_1)\\ 
&   -
    (
    A(\nabla u) 
    -
    A(\nabla \tilde{u}-\beta(\varphi_2)\nabla\varphi_2)
    )
    (\tilde{u}-\vartheta_2) 
    \\  =&
    -(
    A(\nabla \tilde{u}-\beta(\varphi_1)\nabla\varphi_1)
    -
    A(\nabla \tilde{u}-\beta(\varphi_2)\nabla\varphi_2)
    )
    (\tilde{u}-\vartheta_1) 
    \\ & 
    -
    (
    A(\nabla u) 
    -
    A(\nabla \tilde{u}-\beta(\varphi_2)\nabla\varphi_2)
    )
    (\vartheta_1-\vartheta_2) .  
  \end{aligned}
\end{equation}
Maximal regularity for dG (see \eqref{eq:dG-continuous-mr2}) gives
\begin{equation*}
\begin{aligned}
  &\|\widehat\zeta_t\|_{L^p((0,T);L^r(\varOmega))}
  +\|\widehat\zeta\|_{L^p((0,T);W^{2,r}(\varOmega))}
  + \|\zeta\|_{L^p((0,T);W^{2,r}(\varOmega))}
  \leqslant
  C_{p,q,T} 
  \|
  G
  \|_{L^p((0,T);L^r(\varOmega))}
  \\ &\quad
  \leqslant
  C_{p,q,T} 
  \|
  (
    A(\nabla \tilde{u}-\beta(\varphi_1)\nabla\varphi_1)
    -
    A(\nabla \tilde{u}-\beta(\varphi_2)\nabla\varphi_2)
    )
    (\tilde{u}-\vartheta_1) ) 
  \|_{L^p((0,T);L^r(\varOmega))}
  \\ & \qquad 
  +
  \|
  (
    A(\nabla u) 
    -
    A(\nabla \tilde{u}-\beta(\varphi_2)\nabla\varphi_2)
    )
    (\vartheta_1-\vartheta_2) 
  \|_{L^p((0,T);L^r(\varOmega))}
  \\ & \quad \leqslant 
  C_{p,q,T} 
  L
  \|
  \beta(\varphi_1)\nabla\varphi_1 
  - \beta(\varphi_2)\nabla\varphi_2 
  \|_{L^{\infty}((0,T);L^{\infty}(\varOmega))}
  \big(
  \| \tilde{u}\|_{L^p((0,T);W^{2,r}(\varOmega))}
  +
  \| \vartheta_1 \|_{L^p((0,T);W^{2,r}(\varOmega))}
  \big)
  \\ & \qquad 
  +
  C_{p,q,T} 
  L
  (
  \|
  u-\tilde{u}
  \|_{L^{\infty}((0,T);W^{1,\infty}(\varOmega))}
  +
  \|\beta(\varphi_2)\nabla\varphi_2 
  \|_{L^{\infty}((0,T);L^{\infty}(\varOmega))}
  )
  \| \zeta \|_{L^p((0,T);W^{2,r}(\varOmega))}
  \\ & \quad \leqslant 
  C C_{p,q,T} 
  L
  \|
  \nabla(\varphi_1-\varphi_2 )
  \|_{L^{\infty}((0,T);L^{\infty}(\varOmega))}
  +
  C_{p,q,T} 
  L 
  ( M k^{q-1/p} +  \varepsilon ) 
  \| \zeta \|_{L^p((0,T);W^{2,r}(\varOmega))} .
\end{aligned}
\end{equation*}
The constant $C$ depends on $M$ and $L$ but not on $k$ and
$\varepsilon$. We used \eqref{eq:uniform-lipschitz4xx},
\eqref{eq:definitionbeta2cN}, \eqref{eq:definitionbeta2aN}, and that
$ \| \tilde{u}\|_{L^p((0,T);W^{2,r}(\varOmega))}+\| \vartheta_1
\|_{L^p((0,T);W^{2,r}(\varOmega))} \leqslant C$ by \eqref{eq:8} and
\eqref{eq:derror-equation3aN3x}. If
$ C_{p,q,T} L ( Mk^{q-1/p} + \varepsilon ) \leqslant 1/2 $ (the same
smallness condition as before), then we can hide the last term in the
left-hand side and get
\begin{equation}
\label{eq:derror-equation3aN3a}
\begin{aligned}
  &\|\widehat\zeta_t\|_{L^p((0,T);L^r(\varOmega))}
  +\|\widehat\zeta\|_{L^p((0,T);W^{2,r}(\varOmega))}
  + \frac12\|\zeta\|_{L^p((0,T);W^{2,r}(\varOmega))}\\
  &\quad \leqslant 
  C
  C_{p,q,T} 
  L 
  \|
  \nabla(\varphi_1-\varphi_2 )
  \|_{L^{\infty}((0,T);L^{\infty}(\varOmega))}.
\end{aligned}
\end{equation}
By using the space-time Sobolev inequality \eqref{Sobolev-ineq} we
conclude
\begin{equation}
  \label{eq:lipschitzcontinuity}
  \|
  \mathcal{M}(\varphi_1)
  -
  \mathcal{M}(\varphi_1)
  \|_{L^{\infty}((0,T);W^{1,\infty}(\varOmega))}
  \leqslant 
  C
  \|
  \varphi_1-\varphi_2 
  \|_{L^{\infty}((0,T);W^{1,\infty}(\varOmega))}.
\end{equation}
Hence, $\mathcal{M}\colon X\to X$ is (globally Lipschitz) continuous.

\textit{4. Conclusion of the fixed point theorem.}  We have verified
that $\mathcal{M}$ satisfies the assumptions of Schaefer's fixed point
theorem, Lemma~\ref{Le:Schaefer}. We therefore obtain a fixed point
$\vartheta = \mathcal{M}(\vartheta)\in X$.  It is a solution of
\eqref{eq:Appendix.dg-thetaequation4nnNyy} and satisfies the bounds in
\eqref{eq:derror-equation3aN3x} and \eqref{Sobolev-ineq2x2}.

\subsection{Convergence}
\label{subsec:convergence}

So far we have shown the existence of a solution $\vartheta$ to the
truncated error equation \eqref{eq:Appendix.dg-thetaequation4nnNyy},
which is small in the sense that
$\|\vartheta\|_{L^p((0,T);W^{2,r}(\varOmega))} \leqslant
Ck^q+C\varepsilon$. It remains to show an error bound that does not
depend on $\varepsilon$ and to show that $\vartheta$ is so small that
the truncation does not take effect.

We repeat the calculations in the first step of
Subsection~\ref{sss.step2} but with $\mu=1$ and without relying on
\eqref{eq:definitionbeta2aN}.  Instead, we recall from
\eqref{Sobolev-ineq2x2} that
\begin{equation}
  \label{Sobolev-ineq2b1y}
  \| \vartheta \|_{L^\infty ((0,T);W^{1,\infty}(\varOmega) )}
  \leqslant
  C Mk^q+C LM(Mk^{q-1/p}+\varepsilon) 
  \leqslant \nu, 
\end{equation}
where $\nu$ can be made arbitrarily small.

The maximal regularity Lemma~\ref{Le:maxreg-dG} applied to the error
equation \eqref{eq:Appendix.dg-thetaequation4nnN} gives
\begin{equation}
\label{eq:derror-equation3aNx2M}
\begin{aligned}
  &
  \| \partial_k\widehat\vartheta\|_{\ell^p((0,T);L^r(\varOmega))}
  +
\|\widehat\vartheta_t\|_{L^p((0,T);L^r(\varOmega))}
  +
  \|\widehat\vartheta\|_{L^p((0,T);W^{2,r}(\varOmega))}
  + \|\vartheta\|_{L^p((0,T);W^{2,r}(\varOmega))}
  \\ &\qquad 
  \leqslant
  C_{p,r,q}
  \big(
  \|A(\nabla u)\rho\|_{L^p((0,T);L^r(\varOmega))}
  \\ & \qquad \quad
  +
  \|
  (A(\nabla u)-A(\nabla \tilde u-\nabla\vartheta))
  (\tilde{u}-\vartheta)
  \|_{L^p((0,T);L^r(\varOmega))}\big),
\end{aligned}
\end{equation}
where, in view of \eqref{eq:6b},
\begin{equation}
  \label{eq:9Nx2}
  \|A(\nabla u)\rho\|_{L^p((0,T);L^r(\varOmega))}\big ) 
  \leqslant
  Mk^q
\end{equation}
and, by a modification of \eqref{eq:uniform-lipschitz4xx} and
\eqref{eq:6b}, \eqref{eq:6a}, \eqref{Sobolev-ineq2b1y},
\begin{equation}
  \label{eq:10Nx2M}
  \begin{aligned}
    &
    \|
    (A(\nabla u)-A(\nabla \tilde u-\nabla\vartheta))
    (\tilde{u}-\vartheta)
    \|_{L^p((0,T);L^r(\varOmega))}
    \\ & \qquad 
    \leqslant 
    L
    \big(
    \| \nabla(u-\tilde{u}) \|_{L^{p}((0,T);L^{\infty}(\varOmega))}
    +
    \| \nabla\vartheta \|_{L^{p}((0,T);L^{\infty}(\varOmega))}
    \big)
    \| \tilde{u} \|_{L^\infty((0,T);W^{2,r}(\varOmega))} 
    \\ & \qquad \quad
    +
    L
    \big(
    \| \nabla(u-\tilde{u}) \|_{L^{\infty}((0,T);L^{\infty}(\varOmega))}
    +
    \| \nabla\vartheta \|_{L^{\infty}((0,T);L^{\infty}(\varOmega))}
    \big) 
    \| \vartheta \|_{L^p((0,T);W^{2,r}(\varOmega))}  
    \\ & \qquad 
    \leqslant 
    L M^2 k^{q}
    + L M \| \vartheta \|_{L^{p}((0,T);W^{1,\infty}(\varOmega))}
    + L
    \big(
    M k^{q-1/p}+\nu
    \big)
    \| \vartheta \|_{L^p((0,T);W^{2,r}(\varOmega))}  .
  \end{aligned}
\end{equation}
Hence, 
\begin{equation}
\label{eq:derror-equation3aN2x2}
\begin{aligned}
  &
  \|\widehat\vartheta_t\|_{L^p((0,T);L^r(\varOmega))}
  +
  \|\widehat\vartheta\|_{L^p((0,T);W^{2,r}(\varOmega))}
  + \|\vartheta\|_{L^p((0,T);W^{2,r}(\varOmega))}
  \\ & \qquad 
  \leqslant
  C_{p,q,T} \big(
  M  + L M^2\big) k^{q}
  +
  C_{p,q,T} L M \| \vartheta \|_{L^{p}((0,T);W^{1,\infty}(\varOmega))}
  \\ & \qquad \quad
  +
  C_{p,q,T} L
  \big(
  M k^{q-1/p}+\nu
  \big)
  \| \vartheta \|_{L^p((0,T);W^{2,r}(\varOmega))}  .
\end{aligned}
\end{equation}
Using the bound (an easy consequence of the embedding of
${W^{2,r}(\varOmega)}$ into ${W^{1,\infty}(\varOmega)}$ for $r>d;$ see
\cite[Lemma 9.1]{ALL-MC})
\begin{equation}
   \label{eq:gagliardo-nirenberg}
   \|v\|_{W^{1,\infty}(\varOmega)}
   \leqslant
   \gamma \| v\|_{W^{2,r}(\varOmega)}
   +C_{\gamma}
   \|v\|_{L^r(\varOmega)},
   \quad
   \gamma > 0, \ v\in W^{2,r}(\varOmega),
\end{equation}
we obtain, with $\gamma=\nu/M$,
\begin{equation}
\label{eq:derror-equation3aN2x22}
\begin{aligned}
  &
  \| \partial_k\widehat\vartheta\|_{\ell^p((0,T);L^r(\varOmega))}
  +
  \|\widehat\vartheta_t\|_{L^p((0,T);L^r(\varOmega))}
  +
  \|\widehat\vartheta\|_{L^p((0,T);W^{2,r}(\varOmega))}
  +
  \|\vartheta\|_{L^p((0,T);W^{2,r}(\varOmega))}
  \\ & \qquad 
  \leqslant
  C k^q
  +
  C_\nu \| \vartheta \|_{L^{p}((0,T);L^{r}(\varOmega))}
  +
  C_{p,q,T} L
  \big(
  M k^{q-1/p}+2\nu
  \big)
  \| \vartheta \|_{L^p((0,T);W^{2,r}(\varOmega))}  ,
\end{aligned}
\end{equation}
where $C$ and $C_\nu$ depend on $T$ but not on $k$ and $\varepsilon$.
If $C_{p,q,T} L \big( M k^{q-1/p}+2\nu \big)\leqslant 1/2$ (the same
smallness condition as before, but with $2\nu$ instead of
$\varepsilon$), then the last term can be hidden in the left-hand side
to get
\begin{equation}
\label{eq:derror-equation3aN3x2a}
\begin{aligned}
  &
  \| \partial_k\widehat\vartheta\|_{\ell^p((0,T);L^r(\varOmega))}
  +
  \|\widehat\vartheta_t\|_{L^p((0,T);L^r(\varOmega))}
  +
  \|\widehat\vartheta\|_{L^p((0,T);W^{2,r}(\varOmega))}
  + \|\vartheta\|_{L^p((0,T);W^{2,r}(\varOmega))}\\
  & \qquad
  \leqslant
  Ck^q+C_\nu \| \vartheta \|_{L^{p}((0,T);L^{r}(\varOmega))}. 
  \end{aligned} 
\end{equation}

We now apply a discrete Gronwall argument.  The estimate in
\eqref{eq:derror-equation3aN3x2a} is also valid with $T$ replaced by
$t_n,\ n=1,\dotsc,N$,
\begin{equation}
\label{eq:error-equation4n}
\begin{aligned}
  &
  \| \partial_k\widehat\vartheta\|_{\ell^p((0,t_n);L^r(\varOmega))}
  +\|\widehat\vartheta_t\|_{L^p((0,t_n);L^r(\varOmega))}
  + \|\widehat\vartheta\|_{L^p((0,t_n);W^{2,r}(\varOmega))}
  + \|\vartheta\|_{L^p((0,t_n);W^{2,r}(\varOmega))}
  \\ & \qquad
  \leqslant
  C_\nu \|\vartheta\|_{L^p((0,t_n);L^r(\varOmega))}+Ck^q,
  \quad n=1,\dotsc,N,
\end{aligned}
\end{equation}
with the same constants.  Here we need to bound
$\|\vartheta\|_{L^p((0,t_n);L^r(\varOmega))}$ by
$ \| \partial_k\widehat\vartheta\|_{\ell^p((0,t_n);L^r(\varOmega))}$
in such a way that the discrete Gronwall inequality applies.

The inequality
\begin{equation}
\label{eq:triangle-ineq}
\Big (\sum_{m=1}^{\bar n}\Big (\sum_{\ell=1}^m\alpha_\ell\Big )^p\Big )^{1/p}
\leqslant \sum_{m=1}^{\bar n} \Big (\sum_{\ell=1}^m(\alpha_\ell)^p\Big )^{1/p},
\end{equation}
for nonnegative real numbers $\alpha_1,\dotsc,\alpha_{\bar n},$
is an immediate consequence of the triangle inequality for the $\ell^p$-norm on $\R^{\bar n};$
indeed, the left-hand side is the $\ell^p$-norm on $\R^{\bar n}$ of the sum of the vectors 
\begin{equation}
  \label{eq:normofvectors9}
  \begin{pmatrix} 0 \\ \vdots \\ 0\\ \alpha_1\end{pmatrix}
  +\begin{pmatrix} 0 \\ \vdots \\ \alpha_1 \\ \alpha_2\end{pmatrix}
  +\dotsb+\begin{pmatrix} 0 \\ \alpha_1\\ \vdots \\  \alpha_{\bar n-1}\end{pmatrix}
  +\begin{pmatrix} \alpha_1 \\ \alpha_2\\  \vdots \\  \alpha_{\bar n}\end{pmatrix},
\end{equation}
while the right-hand side is the sum of the $\ell^p$-norms of these vectors;
see \cite{KuLL}.

We write $\widehat\vartheta$ in $J_{m-1}$ in the form
\begin{equation}
\label{eq:reccurence-relation}
\widehat\vartheta (t)=k\sum_{\ell=0}^{m-1}\partial_k\widehat\vartheta
(t-\ell k),\quad t\in (t_{m-1},t_m). 
\end{equation}
Notice that any norm of $\partial_k\widehat\vartheta (\cdot-\ell k)$
in the interval $(t_{m-1},t_m)$ coincides with the corresponding norm
of $\partial_k\widehat\vartheta $ in the interval
$(t_{m-\ell-1},t_{m-\ell}).$

Let us denote by $\alpha_\ell$ the discrete
$\ell^p(L^r(\varOmega))$-(semi)norm of the backward difference
quotient $\partial_k\widehat\vartheta$ of $\widehat\vartheta$ in the
interval $(t_{\ell-1},t_\ell),$
\[\alpha_\ell:= \|\partial_k\widehat\vartheta \|_{\ell^p ((t_{\ell-1},t_\ell);L^r(\varOmega) )}.\]
With this notation, \eqref{eq:reccurence-relation} yields
\[
  \|\vartheta \|_{\ell^p ((t_{m-1},t_m);L^r(\varOmega) )}
  =
  \|\widehat\vartheta \|_{\ell^p ((t_{m-1},t_m);L^r(\varOmega))}
  \leqslant
  k\sum_{\ell=1}^m\alpha_\ell;
\]
therefore,
\[
  \|\vartheta \|_{\ell^p ((0,t_n);L^r(\varOmega) )}
  =\|\widehat\vartheta \|_{\ell^p ((0,t_n);L^r(\varOmega) )}
  \leqslant
  k\Big (\sum_{m=1}^n\Big(\sum_{\ell=1}^m\alpha_\ell\Big)^p\Big)^{1/p},
\]
whence, in view of the inequality \eqref{eq:triangle-ineq},
\begin{equation}
\label{eq:estimate1-2023}
\|\vartheta \|_{\ell^p ((0,t_n);L^r(\varOmega) )}
=\|\widehat\vartheta \|_{\ell^p ((0,t_n);L^r(\varOmega) )}
\leqslant k\sum_{m=1}^n \Big (\sum_{\ell=1}^m(\alpha_\ell)^p\Big )^{1/p}.
\end{equation}
With the notation
\[\varTheta_m:=\Big (\sum_{\ell=1}^m(\alpha_\ell)^p\Big )^{1/p},\]
we have
\[\| \partial_k\widehat\vartheta\|_{\ell^p((0,t_n);L^r(\varOmega))}=\varTheta_n\quad\text{and}\quad
\|\vartheta\|_{L^p((0,t_n);L^r(\varOmega))}\leqslant c\|\vartheta \|_{\ell^p ((0,t_n);L^r(\varOmega) )}
\leqslant c k\sum_{m=1}^n \varTheta_m;\]
hence, \eqref{eq:error-equation4n} yields 
\begin{equation}
\label{eq:error-equation4nn}
\begin{aligned}
\varTheta_n+\|\widehat\vartheta_t\|_{L^p((0,t_n);L^r(\varOmega))}
&+ \|\widehat\vartheta\|_{L^p((0,t_n);W^{2,r}(\varOmega))}
+ \|\vartheta\|_{L^p((0,t_n);W^{2,r}(\varOmega))}\\
&\leqslant C_\nu k\sum_{m=1}^n \varTheta_m+Ck^q, 
\end{aligned}
\end{equation}
$n=1,\dotsc,N.$ We assume $C_\nu k\leqslant \frac12$, so that
  $C_\nu k\varTheta_n$ can be moved to the left-hand side. Then, a
discrete Gronwall inequality yields the optimal order estimate
\begin{equation}
\label{eq:error-equation5}
\begin{aligned}
  &
  \| \partial_k\widehat\vartheta\|_{\ell^p((0,T);L^r(\varOmega))}
  +\|\widehat\vartheta_t\|_{L^p((0,T);L^r(\varOmega))}
  \\ & \qquad 
  + \|\widehat\vartheta\|_{L^p((0,T);W^{2,r}(\varOmega))} 
  + \|\vartheta\|_{L^p((0,T);W^{2,r}(\varOmega))}
  \leqslant C_\nu k^q.
\end{aligned}
\end{equation}
The space-time Sobolev argument also gives
\begin{equation}
  \label{eq:spacetimesobolev10}
  \|\vartheta\|_{L^\infty((0,T);W^{1,\infty}(\varOmega))}
  \leqslant
  C_\nu k^q.
\end{equation}

\subsection{Conclusion of the proof} \label{subsec:conclusion}

We need to make sure that
\begin{equation}
  \label{eq:4maxnormbound2a}
  \|\vartheta\|_{L^\infty((0,T);W^{1,\infty}(\varOmega))}\leqslant \varepsilon ,
\end{equation}
so that the truncation does not take effect and $\vartheta$ is thus a
solution of the original error equation
\eqref{eq:error.dg-thetaequation4nx}.  The desired bound
\eqref{eq:4maxnormbound2a} will follow from
\eqref{eq:spacetimesobolev10} after we fix $\nu$ and then choose $k$
sufficiently small.

We have assumed that
\begin{align}
  \label{eq:11-smallness1}
    C L ( Mk^{q-1/p}+\varepsilon) &\leqslant 1/2
   &&\hspace*{-1.5cm} \text{in
    \eqref{eq:dG-Banach2},
    \eqref{eq:derror-equation3aN3x}, 
    \eqref{eq:derror-equation3aN3a},}
  \\ 
  \label{eq:11-smallness3}
       C Mk^q+C LM(Mk^{q-1/p}+\varepsilon) &\leqslant \nu
      &&\hspace*{-1.5cm} \text{in \eqref{Sobolev-ineq2b1y},}
  \\ 
  \label{eq:11-smallness4}
        C L ( M k^{q-1/p}+2\nu)&\leqslant 1/2 
      &&\hspace*{-1.5cm} \text{in 
       \eqref{eq:derror-equation3aN3x2a},}
  \\ 
  \label{eq:11-smallness5}
       C_\nu k&\leqslant 1/2 
       &&\hspace*{-1.5cm}\text{in \eqref{eq:error-equation5},}
\end{align}
where the constants do not depend on $k, \varepsilon$ and $\nu$
(except $C_\nu$ that does depend on $\nu$).  We must choose
$\varepsilon$ and $\nu$ so that these hold for small $k$.

We first fix $\nu = 1/(8CL)$, so that \eqref{eq:11-smallness4} as
well as \eqref{eq:11-smallness5} hold for small $k$.  We then choose
$\varepsilon=k^{q/2}$. Then, by \eqref{eq:spacetimesobolev10},
\begin{align}
  \label{eq:conclusion2}
  \|\vartheta\|_{L^\infty((0,T);W^{1,\infty}(\varOmega))}
  \leqslant 
  C_\nu k^q 
  \leqslant 
  C_\nu k^{q/2} k^{q/2} 
  = 
  C_\nu k^{q/2} \varepsilon
  \leqslant \varepsilon ,
\end{align}
if $C_\nu k^{q/2}\leqslant1$, so that \eqref{eq:4maxnormbound2a} holds
for small $k$.  Moreover,
\begin{align}
  \label{eq:conclusion4}
  CLk^q+CLM(Mk^{q-1/p}+\varepsilon)
  =
  CLk^q+CLM(Mk^{q-1/p}+k^{q/2})
  \leqslant 
  \nu , 
\end{align}
for small $k$, which is \eqref{eq:11-smallness3}.  Similarly,
\begin{align}
  \label{eq:conclusion3}
  C L M (Mk^{q-1/p}+\varepsilon) 
  =
  C L M (Mk^{q-1/p}+k^{q/2}) 
  \leqslant 1/2
\end{align}
for small $k$, which is \eqref{eq:11-smallness1}.

We conclude that there exist $k_0$ and $\nu$ such that, for $0<k<k_0$,
there is a solution $\vartheta$ of
\eqref{eq:error.dg-thetaequation4nx} with
$\| \vartheta \|_{L^\infty ((0,T);W^{1,\infty}(\varOmega) )}\leqslant \nu$
and hence a solution $U=\tilde{u}+\vartheta$ of \eqref{dg} in a
neighborhood of $u$ with radius $2\nu$.  For such $k$ the convergence
estimates \eqref{eq:error-equation5} and \eqref{eq:spacetimesobolev10}
hold.  In view of the approximation property
\eqref{eq:approx-prop-desired1}, this proves the asserted a priori
error estimate \eqref{apriori-estimate1}.  By making the additional
regularity assumption $u\in W^{q+1,p}\big ((0,T);L^r(\varOmega)\big )$
and by using \eqref{eq:approx-prop-desired2} and
\eqref{eq:approx-prop-desired3} we obtain
\eqref{apriori-estimate2}. 

Finally, the a priori error estimate \eqref{apriori-estimate3} is an
immediate consequence of the space-time Sobolev inequality
\eqref{Sobolev-ineq} and the a priori error estimate
\eqref{apriori-estimate2}.

\subsection{Local uniqueness} \label{subsec:uniqueness}

Let $U_1$ and $U_2$ be two solutions of the dG equation \eqref{dg} as
obtained in the previous step of the proof,
Subsection~\ref{subsec:conclusion}.  Thus, both belong to a
neighborhood of $u$ in $L^\infty ((0,T);W^{1,\infty}(\varOmega) )$
with (small) radius $2\nu$. We want to show that $U_1=U_2$, thereby
showing local uniqueness for equation \eqref{dg}.

We begin by noting that there are $C$ and $M_1$ such that
\begin{align}
  \label{eq:unique2b}
  \| u-U_i \|_{L^\infty ((0,T);W^{1,\infty}(\varOmega) )}
  & \leqslant 
    Ck^{q-1/p}, 
  \\ 
  \label{eq:unique2c}
  \| U_i \|_{L^\infty ((0,T);W^{2,r} (\varOmega) )}
  &\leqslant M_1.            
\end{align}
The first bound follows from \eqref{eq:6a} and
\eqref{eq:spacetimesobolev10}. For the second one we use
\eqref{eq:6a}, \eqref{eq:8}, \eqref{eq:error-equation5} and an
inverse inequality:
\begin{align*}
  \| U_i \|_{L^\infty ((0,T);W^{2,r} (\varOmega) )}
  &\leqslant
    \| u \|_{L^\infty ((0,T);W^{2,r} (\varOmega) )}
    +\| \rho_i \|_{L^\infty ((0,T);W^{2,r} (\varOmega) )}
  \\ & \quad 
       +C k^{-1/p}\| \vartheta_i \|_{L^p ((0,T);W^{2,r} (\varOmega) )}
       \leqslant
       M + C k^{q-1/p} \leqslant M_1. 
\end{align*}

By subtracting the two versions of \eqref{dg}, with $f$ depending
only on $\nabla U$, from each other, and linearizing about
$\nabla u$, we see that the difference
$\eta:=U_1-U_2=\vartheta_2-\vartheta_1\in \V_k^{\text{d}}(q-1)$
satisfies $\eta(0)=0$ and
\begin{equation}
  \label{eq:unique3}
  \int_{J_n}   \big[ ( \eta _t ,v )
  - (A(\nabla u)\eta , v ) \big] \, \d t  
  + ( \eta _n^{+}-\eta _n, v_n^{+}) 
  =\int_{J_n}   (G,v ) \, \d t  
  \quad \forall v \in \P(q-1)
\end{equation}
with remainder 
\begin{equation*}
  \label{eq:unique4} 
  G:= \big(A(\nabla u)-A(\nabla U_1)\big)U_1 
  -
  \big(A(\nabla u)-A(\nabla U_2)\big)U_2 ; 
\end{equation*} 
cf.~the derivation of \eqref{eq:error.dg-thetaequation4nx}.
By adding and subtracting $\big(A(\nabla u)-A(\nabla U_1)\big)U_2$,
we obtain
\begin{equation*}
  \label{eq:unique5} 
  G= -\big(A(\nabla u)-A(\nabla U_1)\big)\eta
  -
  \big(A(\nabla U_1)-A(\nabla U_2)\big)U_2 .  
\end{equation*} 
Similarly to \eqref{eq:10Nx2M}, we obtain  (here we use
\eqref{eq:unique2b} and \eqref{eq:unique2c})
\begin{equation*} \label{eq:unique6}
  \begin{split}
    \|G\|_{L^p ((0,T);L^r (\varOmega) )}
    &\leqslant
    L \| u-U_1 \|_{L^\infty ((0,T);W^{1,\infty}(\varOmega))}
    \|\eta\|_{L^p ((0,T);W^{2,r}(\varOmega))}
    \\ & \quad 
    +
    L \| \eta \|_{L^p ((0,T);W^{1,\infty}(\varOmega))}
    \| U_2 \|_{L^\infty ((0,T);W^{2,r}(\varOmega))}
    \\ & 
    \leqslant
    L Ck^{q-1/p} \| \eta \|_{L^p ((0,T);W^{2,r}(\varOmega))}
    +
    L M_1 \| \eta \|_{L^p ((0,T);W^{1,\infty}(\varOmega))} .
  \end{split}
\end{equation*}
In view of \eqref{eq:gagliardo-nirenberg}, we have
\begin{equation*} \label{eq:unique7}
  \|G\|_{L^p ((0,T);L^r (\varOmega) )}
  \leqslant
  L (Ck^{q-1/p}+M_1\gamma) \| \eta \|_{L^p ((0,T);W^{2,r}(\varOmega))}
  +
  L M_1 C_\gamma\| \eta \|_{L^\infty ((0,T);L^r(\varOmega))}.  
\end{equation*}
The maximal regularity Lemma~\ref{Le:maxreg-dG} applied to
equation \eqref{eq:unique3} now gives
\begin{equation*}
  \begin{aligned}
    &
    \| \partial_k\widehat\eta\|_{\ell^p((0,T);L^r(\varOmega))}
    +
    \|\widehat\eta_t\|_{L^p((0,T);L^r(\varOmega))}
    +
    \|\widehat\eta\|_{L^p((0,T);W^{2,r}(\varOmega))}
    + \|\eta\|_{L^p((0,T);W^{2,r}(\varOmega))}
    \\ &\qquad 
    \leqslant
    C_{p,r,q}
    \|G\|_{L^p ((0,T);L^r (\varOmega) )}
    \\ &\qquad 
    \leqslant
    C_{p,r,q} L (Ck^{q-1/p}+M_1\gamma) \| \eta \|_{L^p ((0,T);W^{2,r}(\varOmega))}
    +
    C_{p,r,q} L M_1 C_\gamma\| \eta \|_{L^\infty ((0,T);L^r(\varOmega))}.
  \end{aligned}
\end{equation*}
If $k$ and $\gamma$ are small, then the first term can be hidden in
the left-hand side and the proof is completed by the discrete Gronwall
argument in Subsection~\ref{subsec:convergence}. 

\section{Conditional a posteriori error estimates}\label{Se:aposteriori}
The purpose of this section is the proof of
Theorem~\ref{Theorem2}.  As in Section~\ref{Se:3}, we assume without
loss of generality that $f$ depends only on $\nabla u$, but not
explicitly on $u$. Let 
\begin{equation}
  \label{eq:aposter1}
  R(t):=\wU_t(t)-\nabla\cdot f(\nabla \wU(t)),\quad t\in (0,T], 
\end{equation}
be the residual of the reconstruction $\wU,$ i.e., the amount by which
$\wU$ misses being an exact solution of the differential equation in
\eqref{ivp}. Consider the error $\hat e:=u-\wU.$ Subtracting
\eqref{eq:aposter1} from the differential equation in \eqref{ivp}, we
obtain
\begin{equation}
  \label{eq:aposter2}
  \hat e_t(t)=\nabla\cdot f(\nabla u(t))-\nabla\cdot f(\nabla
  \wU(t))-R(t),
  \quad t\in (0,T]. 
\end{equation}
Using here the operator notation $A(\nabla v),$ see
\eqref{eq:diff-eq2n} and \eqref{eq:notation-A},
we have
\[\hat e_t(t)=A(\nabla u(t))u(t)-A(\nabla \wU(t))\wU(t)-R(t),\]
whence
\begin{equation}
  \label{eq:aposter3}
  \hat e_t(t)=A(\nabla u(t))\hat e(t)+g(t)-R(t),
\end{equation}
with
\[
  g:=
  \big [A(\nabla u)-A(\nabla \wU) \big ]\wU
  =
  \big [A(\nabla u)-A(\nabla (u-\hat e)\big ]
  \big (u-\hat e\big );
\] 
cf.~\eqref{eq:G1}. Considering $g$ a perturbation term, the maximal
regularity property for non-autonomous linear parabolic equations,
Lemma~\ref{Le:maxreg-nonaut}, applied to \eqref{eq:aposter3} leads to
the preliminary estimate
\begin{equation}
  \label{eq:aposter4}
  \|\hat e_t\|_{L^p ((0,t);L^r (\varOmega) )}
  +\|\hat e\|_{L^p ((0,t);W^{2,r} (\varOmega) )}
  \leqslant C\|g\|_{L^p ((0,t);L^r (\varOmega) )}
  +C\|R\|_{L^p ((0,t);L^r (\varOmega) )}
\end{equation}
for $t\in (0,T].$ To estimate $g$ on the right-hand side of
\eqref{eq:aposter4}, we assume that
\begin{equation}
  \label{eq:aposter5}
  \|\hat e\|_{L^\infty ((0,T);W^{1,\infty}(\varOmega)} 
  \leqslant \nu
\end{equation}
with a sufficiently small constant $\nu,$ independent of $k$ and $N$
such that $Nk\leqslant T;$ notice that it is exactly the assumption
\eqref{eq:aposter5} that makes our a posteriori error analysis
\emph{conditional}.  The a priori error estimate
\eqref{apriori-estimate3} shows that it is, in principle, possible
to achieve this condition.

We argue similarly to Subsection~\ref{subsec:convergence}.  By our
regularity assumption on $u$, we may use the bounds stated in
Subsection~\ref{SSe:regularity-assumptions}. In particular, we have
\begin{equation*}
  \| u\|_{L^\infty ((0,T);W^{1,\infty}(\varOmega)}\leqslant M, \quad
  \|\widehat U\|_{L^\infty ((0,T);W^{1,\infty}(\varOmega)}\leqslant M+\nu\leqslant M+1.    
\end{equation*}
Then, the local Lipschitz bound in \eqref{Lipschitz-A} with constant
$L=L_{M+1}$ gives, similarly to \eqref{eq:10Nx2M},
\begin{equation} \label{eq:estimateofg}
  \begin{split}
    \|g\|_{L^p ((0,t);L^r (\varOmega) )}
    &\leqslant
    L \|\nabla\hat e\|_{L^p ((0,t);L^{\infty}(\varOmega))}
    \|u\|_{L^\infty ((0,t);W^{2,r}(\varOmega))}\\ 
    & \quad +
    L \|\nabla\hat e\|_{L^\infty ((0,t);L^{\infty}(\varOmega))}
    \|\hat e\|_{L^p ((0,t);W^{2,r}(\varOmega))} \\ & 
    \leqslant
    L M \|\hat e\|_{L^p ((0,t);W^{1,\infty}(\varOmega))}
    +
    L \nu \|\hat e\|_{L^p ((0,t);W^{2,r}(\varOmega))}.
  \end{split}
\end{equation}
As in \eqref{eq:derror-equation3aN2x22} this leads to
\begin{align*}
  & \|\hat e_t\|_{L^p ((0,t);L^r (\varOmega) )}
    +\|\hat e\|_{L^p ((0,t);W^{2,r} (\varOmega) )}
  \\ & \qquad
       \leqslant 
       C L M \|\hat e\|_{L^p ((0,t);L^r(\varOmega))}
       +
       C L \nu \|\hat e\|_{L^p ((0,t);W^{2,r}(\varOmega))}
       + C\|R\|_{L^p ((0,t);L^r (\varOmega) )}
\end{align*}
and, if $\nu$ is small enough, 
\begin{equation}
  \label{eq:aposter9}
  \|\hat e_t\|_{L^p ((0,t);L^r (\varOmega) )}
  +\|\hat e\|_{L^p ((0,t);W^{2,r} (\varOmega) )}
  \leqslant 
  C \|\hat e\|_{L^p ((0,t);L^r(\varOmega))}
  + C\|R\|_{L^p ((0,t);L^r (\varOmega) )}
\end{equation}
for $t\in(0,T]$ and appropriate constants denoted by $C$.  The desired a
posteriori bound \eqref{aposteriori-estimate1} will now follow by a
(continuous) Gronwall argument.

For this purpose, we must bound
$\|\hat e\|_{L^p ((0,t);L^r (\varOmega) )}$ by
$\|\hat e_t\|_{L^p ((0,t);L^r (\varOmega) )}$ in a way that 
will allow us to apply the Gronwall inequality. Now,
\[
  \hat e(s)
  =\int_0^s\hat e_t(\tau)\, \d \tau,
  \text{ whence } \|\hat e(s)\|_{L^r (\varOmega)}
  \leqslant \int_0^s\|\hat e_t(\tau)\|_{L^r (\varOmega)}\, \d \tau,
\]
and, thus, by H\"older's inequality, with $p'$ the dual of $p,$
\[\|\hat e(s)\|_{L^r (\varOmega)}^p
  \leqslant s^{p/p'}\int_0^s\|\hat e_t(\tau)\|_{L^r (\varOmega)}^p\, \d \tau.\]
Therefore,
\begin{equation}
  \label{eq:aposter10}
  \|\hat e\|_{L^p ((0,t);L^r (\varOmega) )}^p
  \leqslant \int_0^t s^{p/p'}\|\hat e_t\|_{L^p ((0,s);L^r (\varOmega) )}^p\, \d s.
\end{equation}
In view of the equivalence of the $\ell^1$ and $\ell^p$ norms on $\R^2,$
\eqref{eq:aposter9} yields
\begin{equation}
  \label{eq:aposter8}
  \|e_t\|_{L^p ((0,t);L^r (\varOmega) )}^p
  +\|\hat e\|_{L^p ((0,t);W^{2,r} (\varOmega) )}^p
  \leqslant C\|\hat e\|_{L^p ((0,t);L^r (\varOmega) )}^p
  +C\|R\|_{L^p ((0,t);L^r (\varOmega) )}^p.
\end{equation}
Combining \eqref{eq:aposter8} and \eqref{eq:aposter10}, we get
\[\|\hat e_t\|_{L^p ((0,t);L^r (\varOmega) )}^p
  +\|\hat e\|_{L^p ((0,t);W^{2,r} (\varOmega) )}^p
  \leqslant
  C\int_0^t s^{p/p'}\|\hat e_t\|_{L^p ((0,s);L^r(\varOmega) )}^p\, \d s
  +C\|R\|_{L^p ((0,t);L^r (\varOmega) )}^p.\]
We then infer, using Gronwall's inequality, that
\[\|\hat e_t\|_{L^p ((0,t);L^r (\varOmega) )}^p
  +\|\hat e\|_{L^p ((0,t);W^{2,r} (\varOmega) )}^p
  \leqslant C\|R\|_{L^p ((0,t);L^r (\varOmega) )}^p,\]
and obtain the desired a posteriori error estimate
\begin{equation}
  \label{eq:aposter11}
  \|\hat e_t\|_{L^p ((0,t);L^r (\varOmega) )}+\|\hat e\|_{L^p ((0,t);W^{2,r} (\varOmega) )}
  \leqslant C\|R\|_{L^p ((0,t);L^r (\varOmega) )}
\end{equation}
for $t\in(0,T]$.

We show finally that the a posteriori estimator
$\|R\|_{L^p ((0,T);L^r (\varOmega) )}$ is of optimal order $k^q.$
We have
\begin{equation}
  \label{eq:optimal1}
  R(t)=-\hat e_t(t)+A(\nabla u(t))\hat e(t)+g(t);
\end{equation}
see \eqref{eq:aposter3}. Now, the first two terms on the right-hand
side of \eqref{eq:optimal1} are of order $k^q$ in the desired norm in
view of \eqref{apriori-estimate2}. Finally, $g$ is also of order $k^q$
in the desired norm as a consequence of \eqref{eq:estimateofg} and
\eqref{apriori-estimate2}.

\bibliographystyle{amsplain}

\begin{thebibliography}{10}

\bibitem{AF-2003}
R. A. Adams and J. J. F. Fournier, 
\newblock \emph{Sobolev Spaces}, 
2$^\text{nd}$ ed., \newblock Academic Press, Amsterdam, 2003.
\MR{2424078}

\bibitem{AL-linear} 
G. Akrivis and S. Larsson,
\newblock
\emph{A priori and a posteriori error estimates for discontinuous Galerkin time-discrete methods 
via maximal regularity}, 
\newblock \href{https://arxiv.org/abs/2407.15974v1}{arXiv:2407.15974v1}
(Preprint).


\bibitem{ALL-MC} 
G. Akrivis, B. Li, and C. Lubich,
\newblock
\emph{Combining maximal regularity and energy estimates for time
discretizations of quasilinear parabolic equations},
\newblock Math. Comp. \textbf{86} (2017) 1527--1552.
\href{http://dx.doi.org/10.1090/mcom/3228}{DOI  10.1090/mcom/3228}. 
\MR{3626527}
 

\bibitem{AM-SINUM} 
G. Akrivis and Ch. Makridakis,
\newblock
\emph{On maximal regularity estimates for discontinuous Galerkin time-discrete methods},
\newblock SIAM J. Numer. Anal.  \textbf{60} (2022) 180--194.
\href{https://doi.org/10.1137/20M1383781}{DOI  10.1137/20M1383781}. 
 \MR{4368992}


\bibitem{AM-NM}
 G. Akrivis and Ch. G. Makridakis,
 \newblock
\emph{A posteriori error estimates for Radau~IIA methods via maximal regularity},
\newblock Numer. Math.   \textbf{150} (2022) 691--717.
\href{https://doi.org/10.1007/s00211-022-01271-6}{DOI 10.1007/s00211-022-01271-6}. 
  \MR{4393999}
 
 

 \bibitem{E}
 L. C. Evans, 
 \newblock \emph{Partial Differential Equations}, 
 2$^\text{nd}$ ed., \newblock Graduate Studies in Mathematics 19,
American Mathematical Society, Providence, RI, 2010.
\href{https://doi.org/10.1090/gsm/019}{DOI 10.1090/gsm/019}. 
\MR{2597943}

 \bibitem{KK2024}
T. Kashiwabara and T. Kemmochi, 
\newblock
\emph{Discrete maximal regularity for the discontinuous Galerkin
  time-stepping method without logarithmic factor},  
\newblock SIAM J. Numer. Anal.  \textbf{62} (2024) 1638--1659.
\href{https://doi.org/10.1137/23M1580802}{DOI 10.1137/23M1580802}. 
\MR{4775682}


\bibitem{KLL}
B. Kov\'acs, B. Li, and C. Lubich, 
\newblock
\emph{A-stable time discretizations preserve maximal parabolic regularity},
\newblock SIAM J. Numer. Anal.  \textbf{54} (2016) 3600--3624.
\href{https://doi.org/10.1137/15M1040918}{DOI  10.1137/15M1040918}. 
\MR{3582825}

\bibitem{KuLL}
P. C. Kunstmann, B. Li, and C. Lubich, 
\newblock
\emph{Runge--Kutta time discretization of nonlinear parabolic equations studied via discrete maximal parabolic regularity},
\newblock Found. Comput. Math.  \textbf{18} (2018) 1109--1130.
 \href{https://doi.org/10.1007/s10208-017-9364-x}{DOI  10.1007/s10208-017-9364-x}. 
 \MR{3857906}






\bibitem{LV1}
D. Leykekhman and B. Vexler,
\newblock
\emph{Discrete maximal parabolic regularity for Galerkin finite element methods},
\newblock Numer. Math. \textbf{135} (2017) 923--952.
\href{https://doi.org/10.1007/s00211-016-0821-2}{DOI 10.1007/s00211-006-0821-2}. 
\MR{3606467}


\bibitem{LV2}
D. Leykekhman and B. Vexler,
\newblock
\emph{Discrete maximal parabolic regularity for Galerkin finite element methods
for nonautonomous parabolic problems},
\newblock SIAM J. Numer. Anal. \textbf{56} (2018) 2178--2202.
\href{https://doi.org/10.1137/17M114100X}{DOI 10.1137/17M114100X}. 
\MR{3829516}

\bibitem{MN}
Ch. Makridakis and R.~H. Nochetto,
\newblock \emph{A posteriori error analysis for higher order 
dissipative methods for evolution problems},
\newblock Numer. Math. \textbf{104} (2006) 489--514.
\href{https://doi.org/10.1007/s00211-006-0013-6}{DOI 10.1007/s00211-006-0013-6}. 
\MR{2249675} 


\bibitem{SS}
D. Sch\"otzau and C. Schwab,
\newblock
\emph{Time discretization of parabolic problems by the hp-version of the discontinuous Galerkin finite element method},
\newblock SIAM J. Numer. Anal. \textbf{38} (2000) 837--875.
\href{https://doi.org/10.1137/S0036142999352394}{DOI 10.1137/S0036142999352394}. 
\MR{1781206}

\bibitem{T}
V. Thom{\'e}e,
\newblock \emph{Galerkin Finite Element Methods for Parabolic Problems},
2$^\text{nd}$ ed., \newblock Springer--Verlag, Berlin, 2006.
\href{https://doi.org/10.1007/3-540-33122-0}{DOI 10.1007/3-540-33122-0}. 
\MR{2249024}



\end{thebibliography}

\end{document}